\documentclass[12pt]{amsart}

\usepackage{cancel}
\usepackage{latexsym, amssymb, amsmath,esint}
\usepackage{soul}
\usepackage{amsfonts, graphicx}
\usepackage{graphicx,color}

\usepackage{amsmath, amssymb, amsfonts, mathrsfs, mathtools}
\usepackage{latexsym, esint, graphicx, xcolor, color}
\usepackage{wasysym, stmaryrd}
\usepackage{hyperref}
\usepackage[alphabetic]{amsrefs}

\newcommand{\be}{\begin{equation}}
\newcommand{\ee}{\end{equation}}
\newcommand{\beq}{\begin{eqnarray}}
\newcommand{\eeq}{\end{eqnarray}}

\usepackage{wasysym,stmaryrd}
\newtheorem{thm}{Theorem}[section]

\newtheorem{lma}[thm]{Lemma}
\newtheorem{prop}[thm]{Proposition}

\newtheorem{defn}[thm]{Definition}
\theoremstyle{remark}
\newtheorem{rem}[thm]{Remark}
\numberwithin{equation}{section}

\newtheorem{claim}{Claim}[section]

\def\be{\begin{equation}}
\def\ee{\end{equation}}
\def\bee{\begin{equation*}}
\def\eee{\end{equation*}}

\def\lf{\left}
\def\ri{\right}

\mathtoolsset{showonlyrefs}

\def\Ric{\text{\rm Ric}}
\def\Rm{\text{\rm Rm}}

\def\wt{\widetilde}

\def\tr{\operatorname{tr}}

\def\e{\varepsilon}

\def\a{{\alpha}}
\def\b{{\beta}}

\newcommand{\loc}{\mathrm{loc}}
\newcommand{\ADM}{\mathrm{ADM}}
\DeclareMathOperator{\scal}{scal}

\begin{document}

\title[Spaces with lower bound on distributional scalar curvature]{Spaces with distributional scalar curvature bounded from below: Optimal regularity and positive mass}

\author{Man-Chun Lee}
\address[Man-Chun Lee]{Department of Mathematics, The Chinese University of Hong Kong, Shatin, Hong Kong, China}
\email{mclee@math.cuhk.edu.hk}

\author{Florian Litzinger}
\address[Florian Litzinger]{Otto-von-Guericke-Universität Magdeburg, Fakultät für Mathematik, Institut für Analysis und Numerik, Universitätsplatz 2, 39106 Magdeburg, Germany}
\email{florian.litzinger@ovgu.de}

\author{Miles Simon}
\address[Miles Simon]{Otto-von-Guericke-Universität Magdeburg, Fakultät für Ma{-}thematik, Institut für Analysis und Numerik, Universitätsplatz 2, 39106 Mag{-}deburg, Germany}

\email{miles.simon@ovgu.de}


\date{June 22, 2026}

\begin{abstract}
In this work, we study the positive mass theorem under critical low regularity assumptions using Ricci flow smoothing. We show that asymptotically flat manifolds $(M^n,g)$ of regularity $L^\infty\cap W^{1,n}$ with non-negative distributional scalar curvature have non-negative ADM mass. Furthermore, when the ADM mass vanishes, the manifold is globally isometric to Euclidean space with respect to an integral distance introduced by De~Cecco–Palmieri. This extends the recent work of Hafemann to the critical regularity case. Our approach is based on showing that Riemannian metrics of regularity $L^\infty\cap W^{1,n}$, whose scalar curvature is bounded from below in the distributional sense, admit a Ricci flow smoothing whose scalar curvature is bounded from below by the same initial lower bound in the classical sense. In contrast, Cecchini–Frenck–Zeidler constructed examples of metrics which are in $L^\infty\cap W^{1,p}$ for all $2<p<n$, and whose distributional scalar curvature is  bounded from below, that cannot be approximated by smooth metrics with the same scalar curvature lower bound. In this sense, our result is optimal.
\end{abstract}

\maketitle

\section{Introduction}
\label{sec: introduction}

A fundamental problem in Riemannian geometry is to understand the geometric consequences of pointwise curvature conditions. Geometric rigidity plays a central role in distinguishing manifolds from space forms. One of the most basic questions in scalar curvature geometry asks to what extent Euclidean space is rigid under the assumption of non-negative scalar curvature. In this direction, the classical Geroch conjecture asserts that any Riemannian metric on a torus with non-negative scalar curvature must be flat. This conjecture was proved by Schoen–Yau \cites{SchoenYauTorus,SchoenYauTorusII} for dimensions $n\leq 7$ using minimal surface techniques, and by Gromov–Lawson \cite{GromovLawson1980} in all dimensions using spinor methods. In the non-compact setting, the positive mass theorem states that an asymptotically flat manifold with non-negative scalar curvature must have non-negative Arnowitt–Deser–Misner (ADM) mass as introduced in \cite{ADM}. Moreover, if the ADM mass of an end vanishes, then the manifold must be isometric to Euclidean space. This theorem was proved by Schoen–Yau \cites{SchoenYauPMT,SchoenNOTE} for $n\leq 7$, and by Witten \cite{Witten} in the spin case; see also the recent breakthroughs by Bi–Hao–He–Shi–Zhu \cite{PMT19} and Brendle–Wang \cite{BrendleWang}. Through Lohkamp’s compactification method \cite{Lohkamp1999}, the positive mass theorem is in fact equivalent to the non-existence of metrics of positive scalar curvature on connected sums of the form $\mathbb{T}^n \# N$, where $N$ is a closed manifold.

In recent years, there has been growing interest in understanding scalar curvature rigidity under weaker regularity assumptions. Metrics of low regularity arise naturally in stability questions, in the study of the Brown–York quasi-local mass, and in formulations of the positive mass theorem allowing compact singularities. The main objective of this work is to study the scalar curvature rigidity for metrics with minimal regularity assumptions. Motivated in part by the positive mass theorem with corners, much of the literature considers metrics that are smooth away from a compact singular set $\mathcal{S}$, whose size is small in terms of Hausdorff or Minkowski dimension. The case where $\mathcal{S}$ is a hypersurface was first studied by Miao \cite{Miao}, who assumed that the metric is Lipschitz across $\mathcal{S}$ and satisfies mean curvature inequalities. Later, McFeron–Székelyhidi \cite{McFeronSzekelyhidi2012} gave an alternative proof using Ricci flow, establishing rigidity as well. Since then, there has been substantial progress on the interplay between the size of $\mathcal{S}$ and the regularity of the metric across $\mathcal{S}$; see, for example,
\cites{Guim, CecchiniFrenckZeidler2024,ChuLee2025,DaiSunWang2025,DaiWangWangWei,HeShiYu,JiangShengZhang2022,JiangShengZhang2023,Kazaras,Lee2013,LeeTamTAMS,LiMantoulidis,
ShiTam2018} and the references therein. In most of these works, the metric is assumed to lie in $L^\infty$, an assumption closely related to a conjecture of Schoen (see \cites{LiMantoulidis,CecchiniFrenckZeidler2024}).

\begin{defn}\label{defn:Linfty-metric}
Let  $(M,h)$ be  a smooth Riemannian manifold. Then $g$ is said to be a \emph{(globally) $L^\infty$-metric with respect to $h$} if $g$ is a measurable section of $\mathrm{Sym}_2(T^*M)$ such that $$\Lambda^{-1}h\leq g\leq \Lambda g$$ almost everywhere on $M$ for some $\Lambda>1$. We further say that  $g\in W^{1,p}$ with respect to $h$ if $$||\tilde\nabla g||_{L^p(M,h)}<+\infty,$$ where $\tilde\nabla$ denotes the connection of $h$. 
\end{defn}

Throughout the paper, we shall always be working with smooth background Riemannian manifolds $(M^n,h)$ of bounded geometry.
\begin{defn}\label{defn:bdd-geom}
A complete Riemannian manifold $(M,h)$ is said to have \emph{bounded geometry} if there exists a sequence $\{a_k\}_{k=0}^\infty$ and $\iota>0$ such that 
\begin{equation}
  \mathrm{inj}(M,h)>\iota\quad\text{and}\quad   \sup_M |\nabla^{h,k}\mathrm{Rm}(h)|\leq a_k\quad\text{for all}\; k\geq 0.
\end{equation}
\end{defn}

When $M$ is compact, any smooth metric satisfies this condition. This assumption also implies that $(M,h)$ is of infinite volume if $M$ is non-compact. In this work, we focus on the more general setting in which the singular set may have full Hausdorff dimension. This framework was first considered by Lee–LeFloch \cite{LeeLeFloch2015}, who introduced a distributional notion of scalar curvature for metrics in $L^\infty\cap W^{1,2}_{\loc}$.

\begin{defn}[Definition 2.1 in \cite{LeeLeFloch2015}]\label{defn:R-lower-dist}
Let $M$ be a smooth manifold with background metric $h$. A Riemannian metric $g \in W^{1,2}_{\loc}\cap L^\infty$ is said to have \emph{$\scal(g)\geq \kappa$ for $\kappa\in \mathbb{R}$ in the distributional sense} if 
$$\langle \langle \scal(g),\varphi\rangle\rangle \geq \kappa \int_M \varphi \; d\mathrm{vol}_h$$
for all non-negative test function $\varphi\in C^\infty_{c}(M)$, where
\begin{equation*}
\begin{split}
 \langle\langle \scal(g),\varphi\rangle\rangle := \int_M \lf[-\left\langle V,\tilde  \nabla \left(\varphi \cdot \frac{d\mathrm{vol}_{g}}{d\mathrm{vol}_h}\right)\right\rangle_h+ F\varphi  \cdot \frac{d\mathrm{vol}_{g}}{d\mathrm{vol}_h}\,\ri]\; d\mathrm{vol}_h
\end{split}
\end{equation*}
with
\begin{equation}\label{eqn:F-V-dist}
\left\{
\begin{array}{ll}
\Psi^k_{ij}=\frac12 g^{kl} \left(  \tilde\nabla_i g_{kl}+\tilde\nabla_j g_{il}-\tilde\nabla_l g_{ij} \right),\\[3pt]
V^k=g^{ij}\Psi_{ij}^k -g^{ik}\Psi_{ji}^j,\\[3pt]
F=\tr_g \wt\Ric - \Psi_{ij}^k \tilde\nabla_k g^{ij} + \Psi^i_{jl}\tilde \nabla_k g^{ik}+g^{ij} \left(\Psi^k_{kl}\Psi^l_{ij} -\Psi^k_{jl}\Psi^l_{ik} \right).
\end{array}
\right.
\end{equation}
Here $\tilde \nabla$ denotes the connection with respect to the background metric $h$.
\end{defn}

Lee–LeFloch also introduced a notion of ADM mass for metrics in $C^0\cap W^{1,n}_{-q}$, $q\geq \frac{n-2}2$, and proved the positive mass theorem for spin manifolds in this setting, extending Witten’s argument to lower regularity; see also \cites{Bartnik,GrantTassotti,Li}. A notion of ADM mass for metrics with only $C^0$-regularity was also proposed by Burkhardt-Guim \cite{Guim}. A notion of mass has also been introduced by Lundgren–Meco \cite{LundgrenMeco} for metrics with Sobolev regularity $W^{1,2}_{\loc}\cap L^\infty$. While the spin case is now reasonably well understood, the non-spin case at this critical regularity remains subtle.

Recently, Hafemann \cite{Hafemann2026} established a positive mass theorem for non-spin manifolds  $(M^n,g_0)$, $3\leq n\leq 7$, assuming $g_0\in C^0\cap W^{1,n}_{\loc}$ and $g_0$ smooth outside a compact set. Under the stronger assumption $g_0\in C^0\cap W^{1,p}_{\loc}$, $p>n$, rigidity was also proved. One of the  goals of the present work is to extend Hafemann’s result \cite{Hafemann2026}  to the critical regularity class $g_0\in L^\infty\cap W^{1,n}_{\loc}$, including rigidity.

 To formulate rigidity in this non-smooth setting, we employ a generalized notion of distance for $L^\infty$-metrics introduced by De~Cecco–Palmieri \cite{DeCeccoPalmieri}; see also the distance defined and considered by Lamm–Simon in \cite{LammSimon}.

\begin{defn}\label{defn:intr-dist}
Suppose that $g$ is a $L^\infty$-metric on a smooth manifold $M$. Then we define $d_g(x,y):=\lim_{p\to+\infty} d_p(x,y)$, where
\begin{equation}
d_p(x,y)=\left(\inf  \left\{ ||\nabla u||_{L^p(M,g)}: u\in \mathrm{Lip}(M), u(x)=0,u(y)=1 \right\}\right)^{-1}
\end{equation}
for $x,y\in M$.
\end{defn}

The limit exists and defines a distance on $M$ by \cite[Theorem 2.6]{DeCeccoPalmieri} (see also \cite{LeeNaberNeumayer} for development related to scalar curvature stability); moreover it agrees with the Riemannian distance in the smooth case. An equivalent formulation is
\begin{equation}\label{alternativddefn} 
d_p(x,y):=\sup\left\{ |f(x)-f(y)|: f\in W^{1,p}(M), \int_M |\nabla f|^p \,d\mathrm{vol}_g\leq 1\right\}. \\ 
\end{equation}

Using this framework, we prove the following positive mass theorem.

\begin{thm} \label{thm:PMT} 
Suppose that $(M^n, g_0)$, $3\leq n\leq 7$ is a smooth Riemannian manifold and $g_0 \in L^\infty\cap W^{1,n}_{\loc}$, smooth outside a compact set $\Omega$, such that $g_0$ is asymptotically flat in $C^2_\delta$ for $\delta>\frac{n-2}2$ and $\scal(g_0)\in L^1(M\setminus \Omega)$. If $\scal(g_0)\geq 0$ in the distributional sense on $M$, then we have  $$m_{\ADM}(g_0)\geq 0.$$
Furthermore, if $m_{\ADM}(g_0)=0$, then $M$ is diffeomorphic to $\mathbb{R}^n$, $g_0$ is flat outside a compact set, and $g_0$ is globally isometric to the Euclidean space in the sense of Definition~\ref{defn:intr-dist}.
\end{thm}

We refer the reader to \cite[Section 2]{McFeronSzekelyhidi2012} for the definition of the space $C^2_\delta$ and the ADM mass $m_{\ADM}$. Together with \cite[Lemma 2.7]{JiangShengZhang2022}, Theorem~\ref{thm:PMT} also gives an alternative proof to \cite[Theorem 1.1]{JiangShengZhang2022}.

\begin{rem}
In comparison with the method in \cites{Hafemann2026,JiangShengZhang2022}, Theorem~\ref{thm:PMT} is based on smoothing and the dimension restriction is only due to the availability of the positive mass theorem with asymptotically flat ends, but not its proof. In particular, using the recent work of Brendle–Wang \cite[Corollary 1.6]{BrendleWang} (see also \cite{PMT19}), Theorem~\ref{thm:PMT} will hold for all $n\geq 3$ under a slightly stronger asymptotical flatness condition. Indeed, it also holds as long as all other ends of $(M,g)$ are of bounded geometry. In contrast with the general case \cites{LesourdUngerYau2024,Zhu} where incomplete ends are allowed, the singular case remains unclear in general; see \cite{ChuLeeZhu2022} for some related development. 
\end{rem}

We also establish a corresponding rigidity result in the compact case, extending the rigidity \cites{SchoenYauTorus,SchoenYauTorusII,GromovLawson1980} in the smooth case. This can be viewed as a removable singularity result.

\begin{thm}\label{thm:compact}
Let $(\mathbb{T}^n,h)$, $n\geq 3$ be a torus with the standard metric. Suppose that $g_0 \in L^\infty\cap W^{1,n}$ is a Riemannian metric on $\mathbb{T}^n$ such that $\scal(g_0)\geq 0$ in the distributional sense. Then  $g_0$ is isometric to a flat torus 
in the sense of Definition~\ref{defn:intr-dist}.
\end{thm}

\begin{rem}
In Theorem~\ref{thm:PMT} and Theorem~\ref{thm:compact}, if in addition $g_0$ is continuous, then the isometry is in the sense of standard metric spaces, see \cite[Section 3]{DeCeccoPalmieri}. 
\end{rem}

The compact rigidity is a generalization of the work of Jiang–Sheng–Zhang \cite{JiangShengZhang2023} for $W^{1,p}$-metrics with $p>n$; see also \cites{LammSimon,LeeLiu,LitzingerSimon2025,LiSc} for some generalization. These rigidity results stand in sharp contrast to recent examples constructed by Cecchini–Frenck–Zeidler \cite{CecchiniFrenckZeidler2024}, who exhibited $L^\infty\cap W^{2,p/2}$-metrics for $p<n$ that are smooth away from a point, have positive distributional scalar curvature, yet admit no smooth approximation with non-negative scalar curvature. In this sense, our results are optimal.

The technical core of this paper is the proof that an $L^\infty\cap W^{1,n}$-metric with a lower bound on the distributional scalar curvature admits smooth approximations preserving this lower bound, extending the main result in \cite{JiangShengZhang2023} to the complete setting at critical regularity. This smoothing result is independent of compactness or asymptotic flatness and relies on stability estimates for the Ricci flow. 

\begin{thm}\label{thm:smoothing}
Let $(M^n,h)$, $n\geq 3$ be a complete Riemannian manifold with bounded geometry, see Definition~\ref{defn:bdd-geom}. Suppose that $g_0 \in L^\infty\cap W^{1,n}$ is a Riemannian metric on $M$ such that $\scal(g_0)\geq \kappa$ in the distributional sense. Then there exists a family of smooth metrics $g_i$ on $M$ with bounded geometry such that $$\Lambda^{-1}g_0\leq g_i\leq \Lambda g_0,\quad g_i\to g_0 \;\;\mathrm{in}\;\; W^{1,n}_{\loc}, \quad \text{and} \quad \scal(g_i)\geq \kappa$$ for some $\Lambda>1$. Furthermore, the convergence can be improved to $C^0_{\loc}$ if $g_0$ is continuous. 
\end{thm}

In particular, if $g_0$ is a $C^0\cap W^{1,n}$-metric then the scalar curvature in both the sense of approximations \cite{Guim-1} and in distributional sense coincide. Together with \cite[Corollary 1.6]{Guim-0}, it gives an answer to a question in \cite[page 9]{Guim-1}, see the discussion after Question 5.1 there.

\begin{rem}
The sequence of approximations $g_i$ is constructed by taking $g_i:=g(t_i)$ where $g(t)$ is the Ricci–DeTurck flow constructed in \cite{ChuLee2025} (see also \cite{LammSimon}) and $\{t_i\}$ is an arbitrary sequence of times such that $t_i\to 0^+$. This fact is crucial in the applications, that is, the positive mass theorem and torus rigidity. 
\end{rem}

Theorem~\ref{thm:smoothing} is based on showing that the Ricci flow smoothing preserves the scalar curvature lower bound, using an idea motivated by \cite{JiangShengZhang2022},
and the analysis developed by the second and third named author in \cite{LitzingerSimon2025}, see Theorem~\ref{thm:RDF-preserved} for the detailed statement.

\subsection*{Organization of the paper}
Section~\ref{sec:pre} contains preliminary estimates for the Ricci–DeTurck flow.
In Section~\ref{sec:heat}, we derive a priori estimates for the backward heat equation along the flow.
Section~\ref{sec:scalar} establishes preservation of the distributional scalar curvature lower bound and thus proves Theorem~\ref{thm:smoothing}. Finally, Section~\ref{sec:rigidty} proves the positive mass theorem and torus rigidity.

\subsection*{Acknowledgments}
The first named author would like to thank Giorgio Gatti,  Jintian Zhu and Huaiyu Zhang for some useful discussion.  M.-C. Lee is supported by Hong Kong RGC grants No. 14300623 and No. 14304225, and an Asian Young Scientist Fellowship. F. Litzinger and M. Simon were supported by the Special Priority Program SPP 2026 “Geometry at Infinity” of the German Research Foundation (DFG).

\section{Ricci flow smoothing: a priori estimates} \label{sec:pre}
At the heart of the paper lie methods for producing smooth approximations of low regularity metrics using the Ricci flow. For our purposes, we will use the gauged version, i.\,e. Ricci–DeTurck $h$-flow, which is a strictly parabolic flow diffeomorphic to the Ricci flow. More precisely, on a Riemannian manifold $(M,h)$, a smooth family of metrics $g(t)$, $t\in I,$ $I$ an interval, is a solution to the Ricci–DeTurck $h$-flow if it satisfies 
\begin{equation}\label{eqn:RDF-orig}
    \left\{
    \begin{array}{ll}
        \partial_t g_{ij} =-2R_{ij} +\nabla_i W_j+\nabla_j W_i,  \\
         W^k=g^{pq}\left(\Gamma^k_{pq}-\tilde \Gamma^k_{pq} \right),
    \end{array}
    \right.
\end{equation}
for all $t\in I,$
where $\Gamma$ and $\tilde\Gamma$ denote the Christoﬀel symbols of $g(t)$ and $h$, respectively. Equivalently, it satisfies 
\begin{equation}\label{eqn:RDF}
\begin{split}
\partial_t \hat g_{ij}&=\hat g^{pq} \tilde \nabla_p\tilde\nabla_q \hat g_{ij}-\hat g^{kl}\hat g_{ip}\tilde g^{pq} \tilde R_{jkql}-\hat g^{kl}\hat g_{jp} \tilde g^{pq}\tilde R_{ikql}\\[3pt]
&\quad +\frac12 \hat g^{kl}\hat g^{pq}\big(\tilde \nabla_i \hat g_{pk}\tilde \nabla_j\hat  g_{ql}+2\tilde \nabla_k \hat g_{jp}\tilde \nabla_q \hat g_{il}-2\tilde \nabla_k\hat  g_{jp}\tilde \nabla_l \hat g_{iq}\\[3pt]
&\quad -2\tilde \nabla_j\hat  g_{pk}\tilde \nabla_l \hat g_{iq}-2\tilde \nabla_i\hat  g_{pk}\tilde \nabla_l \hat g_{qj}\big),
\end{split}
\end{equation}
see \cite[Lemma 2.1]{Shi1989}. For every Ricci–DeTurck $h$-flow $g(t)$ on a compact manifold, we can find a related Ricci flow
$\hat g(t)$ with $\hat g(t_0) = g(t_0)$ by solving the following system of ordinary differential equations:
\begin{equation}\label{eqn:W-defn}
    \left\{
    \begin{array}{ll}
      \partial_t \Psi_t(x)&=-W(\Psi_t(x),t),\\[3pt]
    \Psi_{t_0}(x)&=x,
   \end{array}
    \right.
\end{equation}
and then setting $\hat g(t):=\Psi_t^*g(t)$.

Assuming an initial metric $g_0$ is in $L^\infty\cap W^{1,n}$, Chu and the first named author \cite{ChuLee2025}, building upon the work of Lamm and the third named author \cite{LammSimon}, established the following existence result. 
\begin{thm}\label{thm:RDF-exist}
Let $(M^n,h)$ be a complete Riemannian manifold with bounded geometry. There exists $\e_0(n,h)>0$ such that if $g_0$ is a $L^\infty\cap W^{1,n}_{\loc}$-metric with respect to $h$ on $M$ so that if
\begin{enumerate}
\item[(i)] $\Lambda_0^{-1}h\leq g_0\leq \Lambda_0 h$ for some $\Lambda_0>1$,
\item[(ii)] for all $x\in M$, 
$$\left(\fint_{B_h(x,1)}|\tilde\nabla g_0|^n d\mathrm{vol}_h \right)\leq \e<\e_0,$$
\end{enumerate}
then there exist $T(n,h),C(n,h),\Lambda(n,\Lambda_0,h)>0$ and a smooth solution $g(t)$ to the Ricci–DeTurck $h$-flow on $M\times (0,T]$ such that
\begin{enumerate}
    \item[(a)] $\Lambda^{-1}h\leq g(t)\leq \Lambda h$,
    \item[(b)] for all $x\in M$, 
$$\left(\fint_{B_h(x,1)}|\tilde\nabla g(t)|^n d\mathrm{vol}_h \right)\leq C\e+Ct^{1/n},$$
    \item[(c)] for any $k\in \mathbb{N}$, there exist $C_k(n,h,\Lambda_0)>0$ such that $$\sup_M |\tilde\nabla^k g(t)|\leq C_kt^{-k/2},$$
    \item[(d)] $g(t)\to g_0$ in $W^{1,n}_{\loc}$ as $t\to 0^+$,
    \item[(e)] $g(t)\to g_0$ in $C^0_{\loc}(\Omega)$ where $\Omega$ is any open set where $g_0$ is continuous.
\end{enumerate}
Furthermore, if $g_0$ is globally in $W^{1,n}$, then for any $\e>0$ and $k\in \mathbb{N}$, there exists $T_{\e,k}>0$ such that, for all $t\in (0,T_{\e,k}]$ and $1\leq \ell\leq k$,
    $$\sup_M |\tilde\nabla^\ell g(t)|\leq \e t^{-\ell}.$$
\end{thm}
\begin{proof}
The existence and the conclusions (a)-(d) follow from \cite[Theorem 1.1]{ChuLee2025}. The higher order asymptotics as $t\to  0$ follow from \cite[Theorem 3.2]{ChuLee2025}.
\end{proof}

Let $g(t)$ be the Ricci–DeTurck $h$-flow from Theorem~\ref{thm:RDF-exist} and fix $t_0\in (0,T]$. We 
will show in Theorem~\ref{thm:RDF-preserved} that $\scal(g(t))\geq \kappa$ in the classical sense if the initial rough metric $g_0$ has the same lower bound in the distributional sense, using the method of  \cite{JiangShengZhang2023}, and \cite{LitzingerSimon2025}, where the method was further developed. In the remainder of this section, and in the following section, we derive some a priori estimates which will be used for this purpose.  

\begin{lma}\label{lma:spacetime-L^2} 
Let $(M^n,h)$ be a complete Riemannian manifold with bounded geometry. For any $\Lambda>1$, there exist $L,\e_1>0$ depending only on $n,h,\Lambda$ such that the following holds: Suppose that $g(t)$ is a solution to the Ricci–DeTurck $h$-flow on $M\times (0,S]$ 
such that \begin{enumerate}
    \item[(i)] $\Lambda^{-1}h\leq g(t)\leq \Lambda h$,
    \item[(ii)] for some $x\in M$, we have
$$\left(\fint_{B_h(x,2)}|\tilde\nabla g(t)|^n d\mathrm{vol}_h \right)^\frac1n\leq \e_1 \;\text{ for all } t\in [0,S].$$
\end{enumerate}
Then for all $s\in (0,S]$ and $p\in [2,n]$, we have  
\begin{equation}
\left\{
\begin{array}{ll}
\displaystyle\int^s_0\fint_{B_h(x,1)} |\Rm(g)|^2 d\mathrm{vol}_h dt\leq L(s+\e_1^2),\\[9pt]
\displaystyle\int^s_0\fint_{B_h(x,1)} |\tilde\nabla g|^{p+2} d\mathrm{vol}_h dt\leq L(s+\e_1^p).
\end{array}
\right.
\end{equation}
\end{lma}
\begin{proof}
The claim follows implicitly from the proof of \cite[Lemma 2.3]{ChuLee2025}. For the reader's convenience we present the details here. We will use $C_i$ to denote any constants depending only on $n,h,\Lambda$. Let $\varphi$ be a cutoff function on $M$ such that $\varphi\equiv 1$ on $B_h(x,1)$, $\varphi$ vanishes outside $B_h(x,2)$, and it satisfies $|\tilde\nabla\varphi|\leq 10$. If we denote $Q:=\sqrt{|\tilde\nabla g|^2+\sigma^2}$ for $\sigma\to 0^+$, then \cite[(2.13)]{ChuLee2025} implies
\begin{equation}\label{eqn:evo-grad}
\left(\partial_t-g^{ij}\tilde\nabla_i \tilde\nabla_j \right) Q^2 \leq -\Lambda^{-1} (|\tilde\nabla^2 g|^2 +|\tilde\nabla g|^4)+C_1 Q^2|\tilde\nabla g|^2+C_1.
\end{equation}

Fix $p\in [2,n]$. By Stokes' theorem and Young's inequality, we have 
\begin{equation}\label{eqn:evo-L^2}
\begin{split}
&\quad \frac1p \int_M \varphi^4 \partial_t Q^p \,d\mathrm{vol}_h+ \Lambda ^{-1}\int_M \varphi^4 Q^{p-2}(|\tilde\nabla^2 g|^2+|\tilde\nabla g|^4)\,d\mathrm{vol}_h\\
&\leq  C_2 \int_M \varphi^4 \left( Q^{p}|\tilde\nabla g|^2+1\right)\, d\mathrm{vol}_h\\
&\leq C_2 \mathrm{Vol}_h(x,2)+C_2 \left(\int_{B_h(x,2)}|\tilde\nabla g|^n d\mathrm{vol}_h\right)^{\frac2n}\cdot \left(\int_M \varphi^\frac{4n}{n-2}Q^\frac{pn}{n-2} d\mathrm{vol}_h\right)^\frac{n-2}{n}\\
&\leq C_2 \mathrm{Vol}_h(x,2)+C_2\e_1^2|\mathrm{Vol}_h(x,2)|^{\frac2n}\cdot \left(\int_M \varphi^\frac{4n}{n-2}Q^\frac{pn}{n-2} d\mathrm{vol}_h\right)^\frac{n-2}{n},
\end{split}
\end{equation}
where we have used assumption (ii).

Since $h$ has bounded geometry and $\varphi$ is compactly supported on $B_h(x,2)$, by Sobolev inequality we deduce that, as $\sigma\to 0$,
\begin{equation}\label{eqn:Sobo-L2}
\begin{split} 
&\quad \mathrm{Vol}_h(x,2)^{\frac2n} \cdot \left(\int_M \varphi^\frac{4n}{n-2}Q^\frac{pn}{n-2} d\mathrm{vol}_h\right)^\frac{n-2}{n}\\
&\leq C_3\left(\int_M \varphi^4|\tilde\nabla Q^{p/2}|^2+\varphi^2 Q^p d\mathrm{vol}_h\right)\\
&\leq C_4\left(\int_M \varphi^4Q^{p-2}|\tilde\nabla Q|^2+\varphi^2 Q^p d\mathrm{vol}_h\right)\\
&\leq C_4\cdot \left(\int_M \varphi^4Q^{p-2}|\tilde\nabla^2 g|^2d\mathrm{vol}_h +\e_1^p\cdot \mathrm{Vol}_h(x,2) \right),
\end{split}
\end{equation}
where we have used $|\tilde\nabla Q|\leq |\tilde\nabla^2 g|$, Hölder's inequality and assumption (ii) to obtain the last inequality.

By \eqref{eqn:evo-L^2}, \eqref{eqn:Sobo-L2}, and Bishop-Gromov volume comparison, we see that if $\e_1$ is small enough, then after passing $\sigma\to0$ we have 
\begin{equation}\label{eqn:space-time-L2-1}
\begin{split}
&\quad  \int^s_0\fint_{B_h(x,1)} |\tilde\nabla g|^{p-2}( |\tilde\nabla g|^4+|\tilde\nabla^2 g|^2 )\,d\mathrm{vol}_h dt\\
&\leq C_5s + \frac1{\mathrm{Vol}_h(x,1)}\int_{B_h(x,2)} |\tilde\nabla g_0|^p d\mathrm{vol}_h\leq C_5(s+ \e_1^p),
\end{split}
\end{equation}
through integration in $t\in [0,s]$. 

To transfer this estimate to curvature $\Rm$ (regarded as a $(3,1)$-tensor), we note that 
\begin{equation}\label{eqn:curv-from-Dg}
\Rm(g)-\Rm(\tilde g)=\tilde\nabla^2 g+ \tilde\nabla g*\tilde\nabla g.
\end{equation}
Together with \eqref{eqn:space-time-L2-1} for $p=2$, this completes the proof.
\end{proof}

\section{A priori estimates for backward heat equation}\label{sec:heat}

In this section, we consider the backward heat equation from $t=t_0$ to $t\approx  0$, starting from a compactly supported function at $t=t_0$. In the following sections, we will denote by $g(t)$ a solution to Ricci–DeTurck $h$-flow on $M\times (0,T]$, where $T\leq 1$ and fix a point $x_0\in M$. We let $K(x,t;y,s)$ be the (minimal) kernel of the operator $\partial_t-\Delta_{x,t}-\nabla_W$, that is, 
\begin{equation}
    \left\{
    \begin{array}{ll}
      (\partial_t-\Delta_{x,t}-\nabla_W)\, K(x,t;y,s)=0,\\[3pt]
        \lim_{t\to s^+} K(x,t;y,s) =\delta_{y}(x),
    \end{array}
    \right.
\end{equation}
and $W$ is the vector field given in \eqref{eqn:RDF}. The attainment of initial data $\delta_y(x)$ is in the sense of distribution. This is the pull-back of the standard heat kernel from the Ricci flow, i.\,e., $\hat g(t):=\Psi_t^*g(t)$. The following heat kernel estimate is based on the work of Bamler, Cabezas-Rivas, and Wilking \cite{BamlerCabezasWilking2019}.

\begin{prop}\label{prop:heat-kernel}
For any $\Lambda>1$, there exists $C(n,h,\Lambda)>0$ such that the following holds: Suppose that $g(t)$ is a solution to the Ricci–DeTurck $h$-flow on $M\times (0,\hat T]$ such that $h$ has bounded geometry and 
\begin{enumerate}
\item[(i)] $\scal(g(t))\geq -\e t^{-1}$,
\item[(ii)] $\Lambda ^{-1}h\leq g(t)\leq \Lambda h$,
\item[(iii)] $|\tilde\nabla g|^2+|\tilde\nabla^2 g|\leq \Lambda t^{-1}$,
\end{enumerate}
for some $\e\in (0,1)$, then the heat kernel $K(x,t;y,s)$ satisfies 
\begin{equation}
    0< K(x,t;y,s)\leq  \left(\frac{t}{s}\right)^\e \frac{C}{(t-s)^{n/2}} \cdot\exp\left(-\frac{d_{h}^2(x,y)}{C(t-s)} \right)
\end{equation}
for all $x,y\in M$ and $0<s<t\leq \hat T$.
\end{prop}
\begin{proof}
It will be slightly more convenient to work with the diffeomorphically related Ricci flow solution $\hat g(t):=\Psi_t^* g(t)$ instead. For $0<s<t\leq \hat T$, we consider the standard heat kernel $\hat K(x,t;y,s)$ and the conjugate heat kernel $\hat G(x,t;y,s)$ with respect to $\hat g(t)$. They satisfy 
\begin{equation}
\left\{
\begin{array}{ll}
(\partial_t -\Delta_{\hat g(t)}) \hat G(x,t;y,s)=\scal_{\hat g(t)}\cdot  \hat G(x,t;y,s),\\[3pt]
(\partial_t -\Delta_{\hat g(t)}) \hat K(x,t;y,s)=0.
\end{array}
    \right.
\end{equation}

Since the Ricci flow $\hat g(t)$ is isometric to the Ricci–DeTurck flow, i.\,e., $g(t) = \Psi_t^*g(t)$, from \eqref{eqn:curv-from-Dg} and the assumptions (ii)-(iii), we have 
\begin{equation}\label{eqn:esti-scaling-invariant}
    |\Rm(\hat g(t))|\leq C_1 t^{-1}\quad\text{and}\quad \mathrm{inj}(\hat g(t))\geq \sqrt{C_1^{-1}t}
\end{equation}
on $M\times (0,\hat T]$ for some $C_1(n,h,\Lambda)>0$. For any $s\in (0,\hat T]$, the function 
\begin{equation}
    F(x,t):= \left( \frac{t}{s}\right)^\e\hat G(x,t;y,s)-\hat K(x,t;y,s)
\end{equation}
satisfies 
\begin{equation}
    \begin{split}
        \left(\frac{\partial}{\partial t}-\Delta_{\hat g(t)} \right) F &=\e s^{-1}\left( \frac{t}{s}\right)^{\e-1} \hat G(x,t;y,s)+\left( \frac{t}{s}\right)^{\e} \scal_{\hat g(t)}\cdot \hat G(x,t;y,s)\\
        &\geq \e t^{-1}\left( \frac{t}{s}\right)^{\e} \hat G(x,t;y,s)-\e t^{-1}\left( \frac{t}{s}\right)^{\e}  \hat G(x,t;y,s)\geq 0.
    \end{split}
\end{equation}
Since $\hat g(t)$ has bounded curvature for $t\in (0,\hat T]$ and  $$\lim_{t\to s^+}\hat G(x,t;y,s)=\delta_y(x)=\lim_{t\to s^+}\hat K(x,t;y,s),$$ it follows from the maximum principle that $F(x,t)\geq 0$ on $M\times (s,\hat  T]$. To be more precise, we let $\phi_0$ be an arbitrary smooth non-negative function with compact support and denote 
\begin{equation}
\left\{
\begin{array}{ll}
  \displaystyle \phi_{\hat G}(x,t) = \int_M \hat G(x,t;y,s) \phi_0(y)\,d\mathrm{vol}_{y,s};\\[4mm]
    \displaystyle  \phi_{\hat K}(x,t) = \int_M \hat K(x,t;y,s) \phi_0(y)\,d\mathrm{vol}_{y,s};\\[3mm]
   \displaystyle F_\phi:= \left( \frac{t}{s}\right)^\e \phi_{\hat G}-\phi_{\hat K}.
\end{array}
\right.
\end{equation}
Then the above computation shows that $F_\phi=0$ at $t=s$ and 
\begin{equation}
 \left(\frac{\partial}{\partial t}-\Delta_{\hat g(t)} \right) F_\phi \geq 0.
\end{equation}
It follows from maximum principle that $F_\phi(x,t)\geq 0$ for all $t\in (s,\hat T]$. Since $\phi_0$ was arbitrary, it follows that $F(x,t)\geq 0$ for all $t\in (s,\hat T]$.

Combining this  with \cite[Proposition 3.1]{BamlerCabezasWilking2019} and \eqref{eqn:esti-scaling-invariant}, we see that there exists a constant $C_0(n,h,\Lambda)>0$ (thanks to $\e\leq 1$) such that
\begin{equation}
\hat K(x,t;y,s)\leq \left(\frac{t}{s}\right)^\e \frac{C_0}{(t-s)^{n/2}} \cdot\exp\left(-\frac{d_{\hat g(s)}^2(x,y)}{C_0(t-s)} \right)
\end{equation}
for all $x,y\in M$ and $0<s< t\leq \hat T$. Using the fact that  $\hat g(t)=\Psi_t^*g(t)$, we conclude 
\begin{equation}\label{eqn:K-est-1}
    \begin{split}
K(x,t;y,s)&=\hat K( \Psi_t^{-1}(x),t; \Psi_s^{-1}(y),s)\\
&\leq \left(\frac{t}{s}\right)^\e \frac{C_0}{(t-s)^{n/2}} \cdot\exp\left(-\frac{d_{g(s)}^2(\Psi_s\circ \Psi_t^{-1}(x),y)}{C_0(t-s)} \right)
    \end{split}
\end{equation}
for all $x,y\in M$ and $0<s< t\leq \hat  T$. 

It remains to estimate $d_{g(s)}(\Psi_s\circ \Psi_t^{-1}(x),y)$. By the triangle inequality with respect to $g(s)$, 
\begin{equation}\label{eqn:dist-RDF-tri}
    \begin{split}
        d_{g(s)}(x,y)&\leq d_{g(s)}(\Psi_s\circ \Psi_t^{-1}(x),y)+d_{g(s)}(\Psi_s\circ \Psi_t^{-1}(x),x).
    \end{split}
\end{equation}
Since $g(s)\leq \Lambda h$ and $|\partial_\tau\Psi_\tau|_h\leq C_0(n,\Lambda) \tau^{-1/2}$, if we denote $z:=\Psi_t^{-1}(x)$, then 
\begin{equation}\label{eqn:dist-RDF-RF-compare}
    \begin{split}
        d_{g(s)}(\Psi_s\circ \Psi_t^{-1}(x),x)&\leq \Lambda^{1/2} d_{h}(\Psi_s(z),\Psi_t(z))\\
        &\leq \Lambda^{1/2} \int^{t}_s |\partial_\tau \Psi_\tau(z)|_h \,d\tau \\
        &\leq C_1(n,\Lambda)(\sqrt{t}-\sqrt{s}).
    \end{split}
\end{equation}

By inserting \eqref{eqn:dist-RDF-tri} and \eqref{eqn:dist-RDF-RF-compare} into \eqref{eqn:K-est-1}, we see that we have shown 
\begin{equation}
    K(x,t;y,s)\leq \left(\frac{t}{s}\right)^\e \frac{C_2}{(t-s)^{n/2}} \cdot\exp\left(-\frac{d_{h}^2(x,y)}{C_2(t-s)} \right)
\end{equation}
 for all $x,y\in M$ and $0<s< t\leq \hat  T,$  for some $C_2(n,\Lambda)>0$.
\end{proof}

\begin{rem}
In practice, if $g(t)$, $t\in (0,T]$ is the solution obtained from Theorem~\ref{thm:RDF-exist}, we may satisfy the assumptions of Proposition~\ref{prop:heat-kernel} by reducing $\hat T<T$ as needed.
\end{rem}

Let $f\in C^\infty_c(B_h(x_0,R_0))$ for some $x_0\in M$ and $R_0>0$. Given $t_0\in (0,T]$, we consider the function $u$,
\begin{equation}\label{defn:u}
    u(y,s):=\int_M K(x,t_0;y,s) f(x)\,d\mathrm{vol}_{g(t_0)}(x),
\end{equation}
which satisfies a backward heat-type equation:
\begin{equation}\label{eqn:heat-f-u}
    \left\{
    \begin{array}{ll}
       \partial_s u=-\Delta_{g(s)}u+\nabla_W u+\scal_{g(s)}\cdot u;\\
       u(t_0)=f.
           \end{array}
    \right.
\end{equation}

\begin{lma}\label{lma:esti-u} 
For $\Lambda>1$, there exists $L_0(n,h,\Lambda)>0$ such that the following holds: Suppose that $g(t)$ is a solution to the Ricci–DeTurck $h$-flow on $M\times (0,\hat T]$ satisfying the assumptions of Proposition~\ref{prop:heat-kernel} for some $\e\in (0,1)$, and let $f\in C^\infty_c(B_h(x_0,R_0))$ for some $x_0 \in M$, $R_0>0$, as well as $t_0 \in (0,\hat T]$ be given.
Then the function $u$ defined in \eqref{defn:u} satisfies  
$$ ||u(s)||_{L^\infty}\leq L_0 s^{-\e}||f||_{L^\infty}$$
for all $s\in (0,t_0]$. Furthermore, if $d_h(x_0,y)\geq 2R_0$, then 
 $$\displaystyle  |u(y,s)|\leq L_0 R_0^{-n} s^{-\e} \exp\left(-\frac{d_{h}^2(x_0,y)}{L_0(t_0-s)}\right)\cdot ||f||_{L^\infty}.$$
\end{lma}
\begin{rem}
The first claim of the lemma in the setting that $M$ is closed and $\hat T \leq 2$ was proved in Theorem 5.3 of \cite{LitzingerSimon2025} using the maximum principle directly. 
\end{rem}

\begin{proof}
In what follows, we will use $C_i$ to denote any constants depending only on $n,h,\Lambda$. 

Since $g(t_0)\leq \Lambda h,$ $t_0\leq 1$, and $f$ is compactly supported on $B_h(x_0,R_0)$ for some large $R_0>1$,  Proposition~\ref{prop:heat-kernel} tells us that 
\begin{equation} \label{eqn:formula-u}
    \begin{split}
        |u(y,s)| &\leq \left(\frac{t_0}{s}\right)^\e \frac{C_1}{(t_0-s)^{n/2}} \int_M  |f(x)|\, \exp\left(-\frac{d_{h}^2(x,y)}{C_1(t_0-s)}\right) d\mathrm{vol}_h(x)\\
        &\leq \frac{C_1 ||f||_{L^\infty}s^{-\e}}{(t_0-s)^{n/2}}\int_{B_h(x_0,R_0)} \exp\left(-\frac{d_{h}^2(x,y)}{C_1(t_0-s)}\right) d\mathrm{vol}_h(x)
    \end{split}
\end{equation}

To estimate the right hand side in \eqref{eqn:formula-u}, it is more convenient to consider the rescaled metric $\tilde h:= (t_0-s)^{-1}h$ where $t_0-s$ is  less than $1$. By the Bishop-Gromov volume comparison and co-area formula, we conclude that
\begin{equation}
\begin{split}
   |u(y,s)|&\leq \frac{C_1  ||f||_{L^\infty}s^{-\e}}{(t_0-s)^{n/2}} \int_{B_h(x_0,R_0)} \exp\left(-\frac{d_{h}^2(x,y)}{C_1(t_0-s)}\right) d\mathrm{vol}_h(x)\\
    &\leq     C_1 ||f||_{L^\infty}s^{-\e} \cdot \int_M \exp\left(-\frac{d_{\tilde h}^2(x,y)}{C_1}\right) d\mathrm{vol}_{\tilde h}(x)\\
       &=    C_1 ||f||_{L^\infty}s^{-\e}\cdot  \int_0^\infty \exp\left(-\frac{r^2}{C_1}\right) \cdot |\partial B_{\tilde h}(y,r)|\cdot dr\\
    &\leq  C_1 ||f||_{L^\infty}s^{-\e}\cdot \int_0^\infty \exp\left( -\frac{r^2}{C_1}+C_2r\right) dr\leq C_3 ||f||_{L^\infty}s^{-\e}.
\end{split}
\end{equation}
This proves the first inequality in the Lemma.

If $d_h(x_0,y)>2R_0$, then \eqref{eqn:formula-u} can be reduced to
\begin{equation} 
    \begin{split}
        |u(y,s) |
        &\leq \frac{C_4||f||_{L^\infty}s^{-\e}}{(t_0-s)^{n/2}} \mathrm{Vol}_h\left(B_h(x_0,R_0)\right)\cdot  \exp\left(-\frac{d_{h}^2(x_0,y)}{C_4(t_0-s)}\right)\\
        &\leq \frac{C_4||f||_{L^\infty}s^{-\e}}{(t_0-s)^{n/2}}\cdot  \exp\left(-\frac{d_{h}^2(x_0,y)}{2C_4(t_0-s)}-\frac{2R_0^2}{C_4(t_0-s)}+C_5R_0\right)\\
        &\leq \frac{C_5||f||_{L^\infty}s^{-\e}}{R_0^n}\cdot  \exp\left(-\frac{d_{h}^2(x_0,y)}{2C_4(t_0-s)} \right),
    \end{split}
\end{equation}
finishing the proof.

\end{proof}

Using the decay rate from Lemma~\ref{lma:esti-u} and the curvature bound for $g(t)$, we see that the integral of the scalar curvature, weighted with a solution to the backwards heat-type equation defined in \eqref{defn:u}, is monotone.

\begin{lma}\label{lma:dist-monot}
Suppose that $g(t)$ is a solution to the Ricci–DeTurck $h$-flow on $M\times (0,\hat T]$ satisfying the assumptions of Proposition~\ref{prop:heat-kernel} for some $\e\in (0,1)$ and let $t_0 \in (0,\hat T]$ be given. Then for all $t\in (0,t_0]$ and $\kappa\in \mathbb{R}$, the function $(\scal_{g(t)}-\kappa)\cdot u(t)$ is integrable on $M$, where $u$ is defined in \eqref{defn:u}. Furthermore, if $f\geq 0$, then
\begin{equation}
    \int_M (\scal_{g(t)}-\kappa)\cdot u(t) \,d\mathrm{vol}_{g(t)}
\end{equation}
is monotonically non-decreasing in $t$.
\end{lma}
\begin{proof}

The integrability of $(\scal_{g(t)}-\kappa)\cdot u(t)$ for $t\in (0,t_0]$ follows directly from the decay estimate of $u$ in Lemma~\ref{lma:esti-u} and the curvature bound, $|\Rm(g(t))|\leq c/t$. It remains to show the monotonicity. When $M$ is compact, this follows from \cite[Proposition 4.1]{JiangShengZhang2023} since $g(t)$ is related  to the Ricci flow by $\Psi_t^*g(t)=\hat g(t),$ where $\Psi$ is a solution to equation \eqref{eqn:W-defn}. We now extend this result to the non-compact case. Let $t_0>s_1>s_2>0$ and $\varphi_R$ be a smooth time-independent cutoff function on $M$ which will be chosen later.

Using equation \eqref{eqn:heat-f-u}, if we denote the operator 
\begin{equation}\label{eqn:label-L}
    \mathfrak{L}:=\partial_t-\Delta_{g(t)}-\nabla_W,
\end{equation}
then $\mathfrak{L}^*=\partial_t+\Delta_{g(t)}-\nabla_W-\scal_{g(t)}\cdot, $ and hence,
\begin{equation}
    \begin{split}
 &\quad  \frac{d}{dt} \int_M  (\scal_{g(t)}-\kappa)\cdot u \varphi_R\,d\mathrm{vol}_{g(t)}\\
      &=\int_M \mathfrak{L}(\scal_{g(t)}-\kappa)\cdot u\varphi_R+ (\scal_{g(t)}-\kappa) \cdot \mathfrak{L}^* (u\varphi_R) \,d\mathrm{vol}_{g(t)}\\
      &=\int_M 2|\Ric(g(t))|^2 u\varphi+ (\scal_{g(t)}-\kappa) \varphi \cdot \mathfrak{L}^*(u)\, d\mathrm{vol}_{g(t)}\\
      &\quad +\int_M  (\scal_{g(t)}-\kappa)\left( u\Delta \varphi_R+ 2\langle \nabla u,\nabla \varphi_R\rangle-u\nabla_W \varphi_R \right) d\mathrm{vol}_{g(t)}\\
      &\geq \int_M  (\scal_{g(t)}-\kappa)\left( u\Delta \varphi_R+ 2\langle \nabla u,\nabla \varphi_R\rangle-u\nabla_W \varphi_R \right) d\mathrm{vol}_{g(t)}\\
      &=:\mathbf{I}+\mathbf{II}+\mathbf{III}.
    \end{split}
\end{equation}

We now construct $\varphi_R$ so that $\varphi_R\to 1$ and $\mathbf{I}, \mathbf{II}$ and $\mathbf{III}$ converge to $0$ as $R\to+\infty$. Fix $x_0\in M$. We let $\rho$ be a smooth proper function on $M$ such that $|\tilde \nabla\rho|^2+|\tilde\nabla^2 \rho|\leq 10$ and $\rho(x)\geq d_h(x_0,x)$, which exists in view of \cite{Tam2010} and  the fact that  $h$ has bounded curvature. Let $\phi$ be a smooth function such that $\phi\equiv 1$ on $[0,1]$, $\phi$ vanishes outside $[0,2]$, and it satisfies $|\phi'|^2+|\phi''|\leq 10^3$.  We let 
\begin{equation}\label{eqn:smooth-cutoff}
    \varphi_R(x):=\phi\left( \frac{\rho(x)}{R}\right),
\end{equation}
which satisfies 
\begin{equation}
    \Delta_{g(t)} \varphi_R=R^{-1}\phi' \Delta_{g(t)} \rho+R^{-2}(\phi')^2|\nabla \rho|^2.
\end{equation}

By writing 
\begin{equation}
    \tilde \nabla_i \tilde\nabla_j \rho=\nabla_i \nabla_j \rho+ (\Gamma_{ij}^k -\tilde \Gamma_{ij}^k) \rho_k
\end{equation}
and using $|\tilde \nabla g|\leq \e t^{-1/2}$ and $\Lambda^{-1}h\leq g(t)\leq \Lambda h$, we have 
\begin{equation}
    \begin{split}
        |\mathbf{I}|\leq C(s_1,s_2,h,n)R^{-2} \int_M u\,d\mathrm{vol}_{g(t)}\to 0
    \end{split}
\end{equation}
as $R\to+\infty$, by Lemma~\ref{lma:esti-u}. The bound of $\mathbf{III}$ is similar. For $\mathbf{II}$, we use  Stokes' theorem so that
\begin{equation}
    \begin{split}
        \mathbf{II}&=-2\int_M u\cdot\mathrm{div}_{g(t)}\left( (\scal_{g(t)}-\kappa)\nabla \varphi\right) d\mathrm{vol}_{g(t)},
    \end{split}
\end{equation}
which also converges to $0$ by $|\nabla \Rm(g(t))|\leq Ct^{-3/2}$ as $R\to+\infty$. Hence, 
\begin{equation}
    \begin{split}
        \int_M  (\scal_{g(t)}-\kappa)\cdot u \varphi_R\,d\mathrm{vol}_{g(t)}\Big|_{t=s_2}^{t=s_1}\geq \int^{s_2}_{s_1}\left(\mathbf{I}+\mathbf{II}+\mathbf{III}\right)\,  ds =o(1)
    \end{split}
\end{equation}
as $R\to+\infty$. The monotonicity now follows from taking $R\to+\infty$ and using the dominated convergence theorem.
\end{proof}

In the following, we show that estimates for $u$ that are uniform in time hold in the $L^p$-setting, complementing the $L^{\infty}$-estimates for $u$ already obtained in Lemma~\ref{lma:esti-u}. 

\begin{prop}\label{prop:L1-u}
Suppose that $g(t)$ is a solution to the Ricci–DeTurck $h$-flow satisfying the assumptions of Proposition~\ref{prop:heat-kernel} for some $\e\in (0,1)$ and let $t_0 \in (0,\hat T]$ be given. Then the function $u$ defined in \eqref{defn:u} satisfies 
$$||u||_{L^1(M,g(t))}=||f||_{L^1(M,g(t_0))},$$
for all $t\in (0,t_0]$.
\end{prop}
\begin{proof}
By Lemma~\ref{lma:esti-u}, $u$ is integrable. We let $\varphi_R$ be the cutoff from \eqref{eqn:smooth-cutoff}. It follows from \eqref{eqn:heat-f-u} and Stokes' theorem that 
\begin{equation}
\begin{split}
\frac{d}{dt} \int_M  u \varphi_R  \,d\mathrm{vol}_{g(t)}&=\int_M \varphi_R\left(-\Delta_{g(t)}u +\nabla_W u+\mathrm{div}_{g(t)}W \cdot u\right) d\mathrm{vol}_{g(t)} \\
&=-\int_M   \left(\Delta_{g(t)} \varphi_R  +\langle W,\nabla \varphi_R \rangle  \right)u \,d\mathrm{vol}_{g(t)}.
\end{split}
\end{equation}

For any $0<t_1<t_2\leq t_0$, we have $|\Delta_{g(t)} \varphi_R  +\langle W,\nabla \varphi_R \rangle |\leq C(n,h,t_1)R^{-1}$ as $R\to +\infty$ thanks to the construction of $\varphi_R$. By integrating it from $t=t_1$ to $t=t_2$, followed by letting $R\to+\infty$, we conclude that $||u||_{L^1(M,g(t))}$ is constant in $t\in (0,t_0]$. 
\end{proof}

We now prove $L^p$-estimates for $u$ which generalize the one obtained in Proposition \ref{prop:L1-u}. 

\begin{prop}\label{prop:Lp-u}
For $\Lambda>1$, there exists $\e_1(n,h,\Lambda)\in (0,1)$ such that the following holds: Suppose that $g(t)$ is a solution to the Ricci–DeTurck $h$-flow on $M\times (0,\hat T]$ satisfying the assumptions of Proposition~\ref{prop:heat-kernel} for some $\e\in (0,1)$, let $t_0 \in (0,\hat T]$ be given, and assume in addition that for some $x\in M$ we have 
\begin{equation}\label{eqn:conc-assump}
\left(\fint_{B_h(x,2)}|\tilde\nabla g(t)|^n d\mathrm{vol}_h \right)^\frac1n\leq \e_1
\end{equation}
for all $t\in (0,\hat T]$.
Then, for any $1<p\leq (4\e)^{-1}$, there exists $L(n,h,\Lambda,p)>0$ such that the function $u$ in \eqref{defn:u} satisfies 
\begin{equation}
\int_{B_h(x,1)}u^p d\mathrm{vol}_{g(t)}\leq  \int_{B_h(x,2)}f^p d\mathrm{vol}_{g(t_0)}+L ||f||_{L^\infty}^p
\end{equation}
for all $t\in  (0,t_0]$. Furthermore, if $d_h(x,x_0)\geq 4R_0$ then we have
\begin{equation}
\int_{B_h(x,1)}u^p d\mathrm{vol}_{g(t)}\leq L  \exp\left(-\frac{d_{h}^2(x_0,x)}{L(t_0-t)}\right)\cdot ||f||_{L^\infty}^p
\end{equation}
for all $t\in (0,t_0]$.
\end{prop}
\begin{proof}
 We will use $C_i$ to denote any constants depending only on $ n,p,\Lambda,h$. Let $\varphi$ be a cutoff function on $M$ such that $\varphi\equiv 1$ on $B_h(x,1)$, $\varphi$ vanishes outside $B_h(x,2)$, and it satisfies $|\tilde\nabla\varphi|^2+|\tilde\nabla^2 \varphi|\leq 10^4$.

We compute the evolution equation of the $L^1$-norm of $\varphi u^p$ using \eqref{eqn:heat-f-u}:
\begin{equation}
\begin{split}
\frac{d}{dt} \int_M \varphi u^p \,d\mathrm{vol}_{g(t)}
&=\int_M pu^{p-1}\varphi \left( -\Delta_{g(t)}u +\nabla_W u\right) \,d\mathrm{vol}_{g(t)}\\
&\quad +\int_M u^p \varphi\left[(p-1) \scal_{g(t)}+ \mathrm{div}_{g(t)}W\right] \,d\mathrm{vol}_{g(t)}\\
&=(p-1)\int_M p u^{p-2} \varphi|\nabla u|^2 + \,\scal_{g(t)}\cdot u^p \varphi\,d\mathrm{vol}_{g(t)}\\
&\quad -\int_M  u^p( \Delta\varphi+\nabla_W \varphi) \,d\mathrm{vol}_{g(t)}.
\end{split}
\end{equation}

Suppose that $L_1>0$ is such that $0\leq u\leq L_1 t^{-\e}$ on $B_h(x,2)\times (0,t_0]$, which exists by Lemma~\ref{lma:esti-u}. By using 
\begin{equation}
\left\{
\begin{array}{ll}
\Delta\varphi= \tr_g \tilde\nabla^2 \varphi+g^{-1}*g^{-1}* \tilde\nabla g* \nabla \varphi,\\[3pt]
W=g^{-1}*\tilde\nabla g,
\end{array}
\right.
\end{equation}
we conclude that 
\begin{equation}
\begin{split}
 \frac{d}{dt} \int_M \varphi u^p \,d\mathrm{vol}_{g(t)}
&\geq -C_1L_1^p t^{-p\e}\int_{B_h(x,2)} |\Rm(g(t))| \,d\mathrm{vol}_{h}\\
&\quad -C_1L_1^p t^{-p\e}\mathrm{Vol}_h(x,2)-C_1 L_1^p t^{-p\e} \int_{B_h(x,2)}   |\tilde\nabla g|\,d\mathrm{vol}_h\\
&\geq -C_1 L_1^p t^{-p\e}\int_{B_h(x,2)}|\Rm(g(t))| \,d\mathrm{vol}_{h} -C_1L_1^p t^{-p\e}\mathrm{Vol}_h(x,2).
\end{split}
\end{equation}
Here we have used \eqref{eqn:conc-assump}.

We integrate in time and also use Lemma~\ref{lma:spacetime-L^2} and Bishop-Gromov volume comparison to conclude that, if $p\e\leq 1/4$, then 
\begin{equation}
\begin{split}
\int_{B_h(x,1)}u^p d\mathrm{vol}_{g(t)}&\leq \int_{B_h(x,2)}f^p d\mathrm{vol}_{g(t_0)}+C_2L_1^p\int^{t_0}_0 s^{-p\e}\,ds\\
&\quad +C_1L_1^p \int^{t_0}_0 s^{-p\e} \int_{B_h(x,2)}|\Rm(g(s))|\,d\mathrm{vol}_h ds\\
&\leq  \int_{B_h(x,2)}f^p d\mathrm{vol}_{g(t_0)}+C_2L_1^p\\
&\quad +C_1L_1^p \int^{t_0}_0  \int_{B_h(x,2)}\left(|\Rm(g(s))|^2+ s^{-2p\e}\right)\,d\mathrm{vol}_h ds\\
&\leq  \int_{B_h(x,2)}f^p d\mathrm{vol}_{g(t_0)}+C_3L_1^p.
\end{split}
\end{equation}

Let $R_0>2$ be such that $f\in C^\infty_c(B_h(x_0,R_0))$. From Lemma~\ref{lma:esti-u}, we may choose $L_1:=L_0(n,\Lambda) \cdot ||f||_{L^\infty}$ to show that 
\begin{equation}
\int_{B_h(x,1)}u^p d\mathrm{vol}_{g(t)}\leq  \int_{B_h(x,2)}f^p d\mathrm{vol}_{g(t_0)}+C_4 ||f||_{L^\infty}^p
\end{equation}
for all $x\in M$.

If $d_h(x,x_0)\geq 4R_0$, then $d_h(y,x_0)\geq 2R_0$ for all $y\in B_h(x,2)$ and $f \equiv 0$ on $B_h(x,2)$. Therefore, from Lemma~\ref{lma:esti-u}, we may choose $$L_1:=L_0\exp\left(-\frac{d_{h}^2(x_0,x_0)}{4L_0(t_0-t)}\right)\cdot ||f||_{L^\infty},$$ so that 
\begin{equation}
\int_{B_h(x,1)}u^p d\mathrm{vol}_{g(t)}\leq C_4  \exp\left(-\frac{pd_{h}^2(x_0,x)}{4L_0(t_0-t)}\right)\cdot ||f||_{L^\infty}^p.
\end{equation}
This completes the proof.
\end{proof}

Next we will derive a local $W^{1,2}$-bound for $u$. We require an interpolation result in Sobolev space, observed by the second and third author \cite{LitzingerSimon2025}; see also Maz'ja \cite[\S 8.2.1]{Mazya1985}:
\begin{lma}\label{lma:interpolation}
There exists $c_1(n)>0$ such that if $(M,g)$ is a complete manifold and $f:M\to\mathbb{R}_{\geq 0}$ is a smooth function with compact support, then we have
\begin{equation}
\int_M \frac{|\nabla f|^4}{f^2}\,d\mathrm{vol}_g\leq c_1\int_M |\nabla^2 f|^2d\mathrm{vol}_g.
\end{equation}
\end{lma}
\begin{proof}
We may use directly the  proof of \cite[Lemma 5.2]{LitzingerSimon2025}. 
There, the fact that $M$ was closed justified the steps involving integration by parts. Here, we use the fact that $f$ has compact support to justify integration by parts. 
\end{proof}

\begin{prop}\label{prop:Wp-u}
For $\Lambda>1$, there exists $\e_1(n,h,\Lambda)\in (0,1)$ such that the following holds. Suppose that $g(t)$ is a solution to the Ricci–DeTurck $h$-flow on $M\times (0,\hat T]$ satisfying the assumptions of Proposition~\ref{prop:heat-kernel} for some $\e\in (0,\frac{n+2}{4(n-2)})$, let $t_0 \in (0,\hat T]$ be given, and assume in addition that for some $x\in M$ we have 
\begin{equation}\label{eqn:conc-assump-1}
\left(\fint_{B_h(z,2)}|\tilde\nabla g(t)|^n d\mathrm{vol}_h \right)^\frac1n\leq \e_1
\end{equation}
for all $t\in (0,\hat T]$ and $z\in B_h(x,2)$. 
Then there exists $L(n,h,\Lambda)>0$ such that the function $u$ in \eqref{defn:u} satisfies 
\begin{equation}
\int_{B_h(x,1)} |\nabla u|^2 \,d\mathrm{vol}_{g(t)}\leq  L||f||^2_{L^\infty}+\int_{B_h(x,2)} |\nabla^{g(t_0)} f|^2 d\mathrm{vol}_{g(t_0)}
\end{equation}
for all $t\in (0,t_0]$. Furthermore, if $d_h(x,x_0)\geq 4R_0$ then we have
\begin{equation}
\int_{B_h(x,1)} |\nabla u|^2 \,d\mathrm{vol}_{g(t)}\leq  L \exp\left(-\frac{d_{h}^2(x_0,x)}{L(t_0-t)}\right)\cdot ||f||_{L^\infty}^2
\end{equation}
for all $t\in (0,t_0]$.

\end{prop}
\begin{proof}
 We will use $C_i$ to denote any constants depending only on $  n,\Lambda,h$. Let $\varphi$ be a smooth cutoff function on $M$ such that $\varphi\equiv 1$ on $B_h(x,1)$, $\varphi$ vanishes outside $B_h(x,2)$, and it satisfies $\varphi^{-1}|\tilde\nabla\varphi|^2+|\tilde\nabla^2 \varphi|\leq 10^4$. We will omit $t$ in $g(t)$ for notational convenience. We also let $L_1>0$ be a constant such that $u\leq L_1 t^{-\e}$ on $B_h(x,2)\times (0,t_0]$.

We compute the evolution equation of the $L^2$-norm of $\varphi^4 |\nabla u|^2$ using  \eqref{eqn:heat-f-u}: 
 \begin{equation}\label{eqn:evo-W^12}
 \begin{split}
 \frac{d}{dt}\int_M \varphi^4 |\nabla u|^2 d\mathrm{vol}_{g}
 &=\int_M \varphi^4 |\nabla u|^2 (-\scal_{g}+\mathrm{div}_{g}W)\,d\mathrm{vol}_{g}\\
 &\quad +\int_M\varphi^4 (2\Ric-\mathcal{L}_Wg)(\nabla u,\nabla u)\,d\mathrm{vol}_{g}\\
 &\quad +\int  2\varphi^4 \langle \nabla (-\Delta u+\nabla_W u +\scal_{g}u),\nabla u \rangle \, d\mathrm{vol}_{g}\\
 &=:\mathbf{I}+\mathbf{II}+ \mathbf{III}.
 \end{split}
 \end{equation}

We estimate $\mathbf{III}$ more carefully. To that end, we split it into three components $\mathbf{III}=:\mathbf{III}_1+\mathbf{III}_2+\mathbf{III}_3$ in the obvious way. By Stokes' theorem, we have 
\begin{equation}\label{eqn:evo-W^12-1}
\begin{split}
\mathbf{III}_3&=-2\int_M \scal_g\, u\cdot  \mathrm{div}_g(\varphi^4 \nabla u)\,d\mathrm{vol}_g\\
&=-2\int_M \scal_g\, u\cdot  (\langle \nabla\varphi^4,\nabla u\rangle+ \varphi^4 \Delta u)\,d\mathrm{vol}_g,
\end{split}
\end{equation}
while by the Bochner formula, we have 
\begin{equation}\label{eqn:evo-W^12-2}
\begin{split}
\mathbf{III}_1&=-2\int_M \varphi^4 \langle \nabla \Delta u ,\nabla u\rangle \,d\mathrm{vol}_g\\
&=\int_M \varphi^4\left[-\Delta |\nabla u|^2+2|\nabla^2 u|^2+2\Ric(\nabla u,\nabla u)\right]\,d\mathrm{vol}_g\\
&=\int_M \varphi^4\left(2|\nabla^2 u|^2+2\Ric(\nabla u,\nabla u)\right)+ \langle \nabla\varphi^4, \nabla |\nabla u|^2\rangle\,d\mathrm{vol}_g.
\end{split}
\end{equation}

To handle $\mathbf{III}_2$, we note that
\begin{equation}\label{eqn:evo-W^12-3}
\begin{split}
&\quad -(\mathcal{L}_W g)(\nabla u,\nabla u)+ 2\langle \nabla \nabla_W u,\nabla u\rangle+\mathrm{div}_g W\cdot  |\nabla u|^2\\
&= -2u_i u_j\nabla^i W^j+ 2u_i \nabla^i (W^ju_j)+\mathrm{div}_g W\cdot  |\nabla u|^2\\
&= 2u_i W^j \nabla^i u_j+\mathrm{div}_g W\cdot  |\nabla u|^2\\
&= \mathrm{div}_g \left(  |\nabla u|^2W\right).
\end{split}
\end{equation}

Once again, by Stokes' theorem, we combine  \eqref{eqn:evo-W^12}, \eqref{eqn:evo-W^12-1}, \eqref{eqn:evo-W^12-2} and \eqref{eqn:evo-W^12-3} to conclude that
\begin{equation}\label{eqn:enert}
\begin{split}
\frac{d}{dt}\int_M \varphi^4 |\nabla u|^2 d\mathrm{vol}_{g}
 &=\int_M -\varphi^4 |\nabla u|^2\scal_g+4\varphi^4\cdot  \Ric(\nabla u,\nabla u)d\mathrm{vol}_g\\
 &\quad +\int_M 2\varphi^4\left(|\nabla^2 u|^2-\scal_g u\cdot  \Delta_g u\right)d\mathrm{vol}_g\\
 &\quad+\int_M \langle \nabla\varphi^4, -\,\scal_g \nabla u^2+\nabla|\nabla u|^2-|\nabla u|^2W \rangle d\mathrm{vol}_g.
\end{split}
\end{equation}

We will apply the Cauchy–Schwarz inequality repeatedly to simplify the energy inequality. We first control the terms involving curvature. For any $1>\delta>0$, we have
\begin{equation}\label{eqn:esti-involving-RM}
    \begin{split}
        &\quad \left|\int_M 4\varphi^4 \Ric(\nabla u ,\nabla u) -\mathrm{scal}_g \cdot \left( \varphi^4 |\nabla u|^2+2\varphi^4 u \cdot \Delta _g u +\langle \nabla \varphi^4 ,\nabla u^2\rangle  \right) d\mathrm{vol}_g \right|\\
        &\leq C_n \int_M \varphi^4 |\Rm|  \left(|\nabla u|^2 +u |\nabla^2u |+\frac{|\nabla\varphi|^2}{\varphi^2} u^2\right)d\mathrm{vol}_g\\
        &\leq \delta  \int_M  \left(\frac{|\nabla u|^4}{u^2}+ |\nabla^2 u|^2  \right)\varphi^4 d\mathrm{vol}_g   \\
        &\quad +C_n\delta^{-1} \int_M  \varphi^2 u^2(|\Rm|^2+1) \left(\varphi^2  +  |\nabla\varphi|^2 \right)d\mathrm{vol}_g.
    \end{split}
\end{equation}

We next control the inner product term. For any $1>\delta>0$,  we also have
\begin{equation}\label{eqn:innter}
    \begin{split}
    &\quad \left|\int_M \langle \nabla\varphi^4 , \nabla |\nabla u|^2-W|\nabla u|^2 \rangle \,d\mathrm{vol}_g\right|\\
    &\leq C_n\int_M \varphi^3 |\nabla \varphi|\left( |\nabla u||\nabla^2u|+|\nabla u|^2 |\tilde\nabla g|\right) \,d\mathrm{vol}_g\\
    &\leq \int_M \delta |\nabla^2 u|^2 \varphi^4+ C_n\delta^{-1}\varphi^2 |\nabla\varphi|^2|\nabla u|^2+C_n \varphi^3 |\nabla\varphi||\nabla u|^2|\tilde\nabla g|  \,d\mathrm{vol}_g
    \\
    &\leq \delta\int_M \left(\frac{|\nabla u|^4}{u^2} +|\nabla^2 u|^2  \right)\varphi^4 d\mathrm{vol}_g \\
    &\quad +C_n\delta^{-3}\int_M \left(|\nabla \varphi|^4 u^2+\varphi^2 u^2 |\tilde\nabla g|^2|\nabla\varphi|^2  \, \right)d\mathrm{vol}_g. 
    \end{split}
\end{equation}

By substituting \eqref{eqn:esti-involving-RM} and \eqref{eqn:innter} into \eqref{eqn:enert}, we arrive at the following inequality:
\begin{equation}\label{eqn:energy-control}
\begin{split}
\frac{d}{dt}\int_M \varphi^4 |\nabla u|^2 d\mathrm{vol}_{g}
 &\geq \int_M \varphi^4 \left(|\nabla^2 u|^2 -\delta \frac{|\nabla u|^4}{u^2} \right) \,d\mathrm{vol}_g\\
 &\quad -C_2\delta^{-1}\int_M \varphi^2(|\tilde\nabla\varphi|^2+\varphi^2)( |\Rm(g)|^2+1) u^2 \,d\mathrm{vol}_g\\
 &\quad -C_2\delta^{-3}\int_M u^2 |\tilde\nabla\varphi|^4+u^2 \varphi^2|\tilde\nabla\varphi|^2 |\tilde\nabla g|^2\,d\mathrm{vol}_g
\end{split}
\end{equation}
for any $1>\delta>0$.

We now use interpolation to squeeze more positivity from the first term. By applying  Lemma~\ref{lma:interpolation} to $f=\varphi^2 u$, we have 
\begin{equation}
\begin{split}
\int_M \frac{|\nabla (\varphi^2 u)|^4}{(\varphi^2 u)^2} d\mathrm{vol}_g\leq c_1 \int_M |\nabla^2 (\varphi^2 u)|^2 d\mathrm{vol}_g.
\end{split}
\end{equation}

Hence, the Cauchy–Schwarz inequality implies
\begin{equation}
\begin{split}
\int_M \varphi^4 \frac{|\nabla u|^4}{u^2}d\mathrm{vol}_g &\leq C_3\int_M \frac{|\nabla (\varphi^2 u)|^4}{(\varphi^2 u)^2} +u^2 |\nabla\varphi|^4 d\mathrm{vol}_g\\
&\leq C_3 c_1 \int_M |\nabla^2 (\varphi^2 u)|^2+u^2 |\nabla\varphi|^4\,d\mathrm{vol}_g\\
&\leq C_4\int_M \varphi^4|\nabla^2 u|^2 +\varphi^2|\nabla^2\varphi|^2 u^2 \,d\mathrm{vol}_g\\
&\quad +C_4\int_M\varphi^2|\nabla\varphi|^2 |\nabla u|^2+u^2 |\nabla\varphi|^4\,d\mathrm{vol}_g\\
&\leq  C_5\int_M \varphi^4|\nabla^2 u|^2+\varphi^2 |\tilde\nabla^2 \varphi|^2u^2+\varphi^2|\tilde\nabla \varphi|^2|\tilde\nabla g|^2 u^2 d\mathrm{vol}_g\\
&\quad +\int_M  \frac12\varphi^4 \frac{|\nabla u|^4}{u^2} +C_5 u^2|\tilde\nabla\varphi|^4 \,d\mathrm{vol}_g.
\end{split}
\end{equation}
Here we have used $\nabla^2\varphi=\tilde\nabla^2 \varphi+g^{-1}*g^{-1}*\tilde\nabla g*\tilde\nabla\varphi$.

Using also the estimate for $\varphi$, this shows that 
\begin{equation}\label{eqn:interpolation-with-err}
\begin{split}
\int_M \varphi^4 \frac{|\nabla u|^4}{u^2}d\mathrm{vol}_g &\leq C_6\int_M \left(\varphi^4 |\nabla^2 u|^2+\varphi^2 |\tilde\nabla g|^2 u^2\right) \,d\mathrm{vol}_g  +C_6 L_1^2 t^{-2\e}\\
&\leq C_6\int_M  \varphi^4 |\nabla^2 u|^2 \,d\mathrm{vol}_g  +C_7 L_1^2 t^{-2\e},
\end{split}
\end{equation}
where we have used \eqref{eqn:conc-assump-1}.

We combine \eqref{eqn:conc-assump-1}, \eqref{eqn:interpolation-with-err}, \eqref{eqn:energy-control} with $\delta := (2C_6)^{-1}$, so that 
\begin{equation}\label{eqn:test-grad-u}
\begin{split}
&\quad \frac{d}{dt}\int_M \varphi^4 |\nabla u|^2 d\mathrm{vol}_{g}\\
&\geq  -C_8L_1^2 t^{-\frac12 -2\e}-C_8\int_{M}\varphi^2 u^2(|\Rm(g)|^2+1)\,d\mathrm{vol}_h+\frac12 \int_M \varphi^4 |\nabla^2 u|^2\,d\mathrm{vol}_g.
\end{split}
\end{equation}

\medskip

We now prepare a barrier function. Recall from \eqref{eqn:evo-grad} that 
\begin{equation}
\begin{split}
\mathfrak{L}(|\tilde\nabla g|^2)&:= \left(\partial_t-\Delta_g-\nabla_W\right)|\tilde\nabla g|^2 \\
&\leq g^{ij} \Psi_{ij}^k \tilde\nabla_k|\tilde\nabla g|^2-\nabla_W |\tilde\nabla g|^2-\Lambda^{-1} |\tilde\nabla^2 g|^2+C_9 |\tilde\nabla g|^4+C_9\\
&= -\Lambda^{-1}|\tilde\nabla^2 g|^2+C_9|\tilde\nabla g|^4+C_9,
\end{split}
\end{equation}
where $\mathfrak{L}$ is defined in \eqref{eqn:label-L} and we have used $\Psi_{ij}^k:=\Gamma_{ij}^k-\tilde\Gamma_{ij}^k$ and $W^k=g^{ij}\Psi_{ij}^k$. Together with 
\begin{equation}
\mathfrak{L}^* u^2:=(\partial_t+\Delta_g-\nabla_W -\scal_g)u^2=2|\nabla u|_g^2+\scal_g  u^2,
\end{equation}
we arrive at 
\begin{equation}\label{eqn:evo-grad-wth-u}
\begin{split}
\frac{d}{dt}\int_M \varphi^2 |\tilde\nabla g|^2 u^2\,d\mathrm{vol}_g
&=\int_M \mathfrak{L} (|\tilde\nabla g|^2)  \cdot \varphi^2 u^2 + |\tilde\nabla g|^2 \cdot \mathfrak{L}^*(\varphi^2 u^2)\,d\mathrm{vol}_g\\
&=\int_M \mathfrak{L} (|\tilde\nabla g|^2)  \cdot \varphi^2 u^2 + |\tilde\nabla g|^2 \varphi^2\cdot \mathfrak{L}^*( u^2)\,d\mathrm{vol}_g\\
&\quad +\int_M |\tilde\nabla g|^2 \big(2\langle\nabla\varphi^2,\nabla u^2\rangle \\
&\hspace{3em} +\scal_g \,\varphi^2 u^2+u^2\mathfrak{L}^* \varphi^2 \big)\,d\mathrm{vol}_g\\
&\leq \int_M \left(-\Lambda^{-1}|\tilde\nabla^2 g|^2+C_9|\tilde\nabla g|^4+C_9 \right)\varphi^2 u^2\,d\mathrm{vol}_g\\
&\quad +\int_M2 \varphi^2|\tilde\nabla g|^2 (|\nabla u|^2+\,\scal_g\,  u^2)\,d\mathrm{vol}_g\\
&\quad +\int_M |\tilde\nabla g|^2 \left(4u\langle\nabla \varphi^2,\nabla u\rangle+u^2(\Delta \varphi^2-\nabla_W\varphi^2) \right)\,d\mathrm{vol}_g\\
&=:\mathbf{IV}+\mathbf{V}+\mathbf{VI}.
\end{split}
\end{equation}

We first handle $\mathbf{V}$. If we write $\scal_g=\tr_g \widetilde\Ric+\tilde\nabla g*\tilde\nabla g+\tilde\nabla^2g$, then
\begin{equation}\label{eqn:I-evo-grad-wth-u}
\begin{split}
&\quad 2 \varphi^2|\tilde\nabla g|^2 (|\nabla u|^2+\,\scal_g\,  u^2)\\
&\leq \frac12 \Lambda^{-1} |\tilde\nabla^2 g|^2\varphi^2u^2+C_{10}\delta^{-1} (1+|\tilde\nabla g|^4) u^2\varphi^2+\delta \frac{|\nabla u|^4}{u^2}\varphi^2
\end{split}
\end{equation}
for any $\delta>0$. 

For $\mathbf{VI}$, we use $\Delta=\tr_g \tilde\nabla^2+(g^{-1}*g^{-1}*\tilde\nabla g )\ \tilde \nabla$ and Young's inequality repeatedly to obtain
\begin{equation}\label{eqn:II-evo-grad-wth-u}
\begin{split}
\mathbf{VI}&\leq C_{11}\int_{B_h(x,2)}u^2\,d\mathrm{vol}_g+C_{11}\int_M \left(u^2\varphi^2 |\tilde\nabla g|^4+\varphi^2|\tilde\nabla g|^2 |\nabla u|^2\right)\,d\mathrm{vol}_g\\
&\leq C_{12}\delta^{-1}\int_{B_h(x,2)}u^2\,d\mathrm{vol}_g+\int_{B_h(x,2)} C_{12}\delta^{-1} u^2 |\tilde\nabla g|^4+\delta \varphi^4\frac{|\nabla u|^4}{u^2}\,d\mathrm{vol}_g\\
&\leq C_{12}\delta^{-1}L_1^2 t^{-2\e} +\delta \int_M\varphi^4\frac{|\nabla u|^4}{u^2}\,d\mathrm{vol}_g+C_{12}\delta^{-1}\int_{B_h(x,2)} u^2 |\tilde\nabla g|^4 d\mathrm{vol}_g
\end{split}
\end{equation}
for any $\delta>0$. 

By inserting \eqref{eqn:I-evo-grad-wth-u} and \eqref{eqn:II-evo-grad-wth-u} into \eqref{eqn:evo-grad-wth-u}, we deduce that 
\begin{equation}\label{eqn:test-grad-u-0}
\begin{split}
\frac{d}{dt}\int_M \varphi^2 |\tilde\nabla g|^2 u^2\,d\mathrm{vol}_g&\leq C_{13} \delta^{-1} L_1^2t^{-2\e}-C_{13}^{-1}\int_M (|\Rm(g)|^2+1) \varphi^2 u^2\,d\mathrm{vol}_g\\
&\quad +C_{13}\delta^{-1}\int_{B_h(x,2)} |\tilde\nabla g|^4  u^2 d\mathrm{vol}_g+\delta \int_M \varphi^4 \frac{|\nabla u|^4}{u^2}d\mathrm{vol}_g.
\end{split}
\end{equation}
Here we have used \eqref{eqn:curv-from-Dg} to get 
\begin{equation}
    |\Rm(g)|^2+1 \leq C_n\left( |\Rm(\tilde g)|+1+|\tilde\nabla g|^4+|\tilde\nabla^2g|^2 \right).
\end{equation}
in order to replace $|\tilde\nabla^2 g|^2$ by $ |\Rm(g)|^2+1 $ with errors.

\medskip

Now we combine \eqref{eqn:test-grad-u-0}, \eqref{eqn:interpolation-with-err} and \eqref{eqn:test-grad-u}, with $\delta=(2C_6)^{-1}$ to deduce that 
\begin{equation}\label{eqn:ODE-energy}
\begin{split}
&\quad \frac{d}{dt}\int_M \left(\varphi^4 |\nabla u|^2- C_8^{-1}C_{13}\varphi^2 |\tilde\nabla g|^2 u^2\right)d\mathrm{vol}_g\\
&\geq -C_{14} L_1^2 t^{-2\e}-C_{14}\int_{B_h(x,2)} u^2|\tilde\nabla g|^4\,d\mathrm{vol}_h.
\end{split}
\end{equation}

Thanks to the range of $\e$, Lemma~\ref{lma:spacetime-L^2} and a simple covering argument imply  
\begin{equation}\label{eqn:space-time-Dg-u}
\begin{split}
&\quad \int^{t_0}_0\int_{B_h(x,2)} u^2 |\tilde\nabla g(s)|^4 \, d\mathrm{vol}_hds\\ &\leq \left(\int^{t_0}_0\int_{B_h(x,2)} |\tilde\nabla g(s)|^{n+2} \, d\mathrm{vol}_hds\right)^\frac{4}{n+2}\cdot \left(\int^{t_0}_0\int_{B_h(x,2)} u^\frac{2(n+2)}{n-2} \, d\mathrm{vol}_hds \right)^\frac{n-2}{n+2}\\
&\leq C_{15} L_1^2.
\end{split}
\end{equation}

Using \eqref{eqn:space-time-Dg-u} and \eqref{eqn:conc-assump-1}, 
we may integrate \eqref{eqn:ODE-energy} from $t\in (0,t_0]$ to $t_0$, which yields
\begin{equation}
\begin{split}
&\quad \int_{B_h(x,1)} |\nabla u|^2 \,d\mathrm{vol}_{g(t)}\leq C_{16}L_1^2+\int_{B_h(x,2)} |\nabla^{g(t_0)} f|^2 d\mathrm{vol}_{g(t_0)}.
\end{split}
\end{equation}

This completes the proof by choosing $L_1$ according to the size of $d_h(x,x_0)$ by Lemma~\ref{lma:esti-u}, as in the proof of Proposition~\ref{prop:Lp-u}.
\end{proof}

\section{Preservation of scalar curvature bounds}\label{sec:scalar}

Using the Ricci flow smoothing from Theorem~\ref{thm:RDF-exist}, we show that if the initial metric $g_0$ has $\scal(g_0)\geq \kappa$ in the distributional sense, then $\scal(g(t))\geq \kappa$ for all $t\in (0,T]$. 

\begin{thm}\label{thm:RDF-preserved} 
Let $g_0$ be a $L^\infty\cap W^{1,n}$-metric on $M$ with respect to a metric $h$ with bounded geometry and let $g(t)$ be the Ricci–DeTurck $h$-flow obtained from Theorem~\ref{thm:RDF-exist}. Suppose that $\scal(g_0)\geq \kappa$ in the distributional sense. Then $\scal(g(t))\geq \kappa$ for all $t\in (0,T]$. 
\end{thm}

\begin{proof}
Since $g(t)$ is smooth with bounded geometry for $t\in (0,T]$, it suffices to show that for some $\hat T>0$ we have $\scal(g(t))\geq \kappa$ on $M\times (0, \hat T]$. We fix $\e>0$ small enough such that $\frac{2n}{n-2}<(4\e)^{-1}$ and let $\hat T(\e,h,n,\Lambda)>0$ such that the assumptions of Proposition~\ref{prop:Lp-u} are satisfied on $(0,\hat T]$ by Theorem~\ref{thm:RDF-exist}.

Let $t_0\in (0,\hat T]$ and let $f$ be a compactly supported smooth non-negative function and $u$ the function obtained from \eqref{eqn:heat-f-u}. By Lemma~\ref{lma:dist-monot} and $u(t_0)=f$, we have 
\begin{equation}\label{eqn:scal-gt-1}
     \int_M (\scal_{g(t_0)}-\kappa)\cdot f \,d\mathrm{vol}_{g(t_0)} \geq \liminf_{t\to0^+}  \int_M (\scal_{g(t)}-\kappa)\cdot u(t) \,d\mathrm{vol}_{g(t)}.
\end{equation}

It suffices to show that the limit on the right hand side is non-negative. Let $R_0>0$ and $R_m:=2^m R_0$, where $m\in \mathbb{N}$. Denote $R:=R_m$ for convenience. Moreover, we let $\varphi_R$ be the cutoff function on $M$ constructed in the proof of Lemma~\ref{lma:dist-monot}. We write $u=u\varphi_R+ u(1-\varphi_R)=:u_R+\hat u_R$. For $t\to 0^+$, we split the scalar curvature integral according to
\begin{equation}\label{eqn:scal-gt}
\begin{split}
   &\quad  \int_M (\scal_{g(t)}-\kappa)\cdot u \,d\mathrm{vol}_{g(t)}\\
   &=\int_M (\scal_{g(t)}-\kappa)\cdot u_R \,d\mathrm{vol}_{g(t)}+\int_M (\scal_{g(t)}-\kappa)\cdot \hat u_R \,d\mathrm{vol}_{g(t)}\\
    &= \langle\langle \scal_{g(t)}-\kappa, u_R\rangle\rangle_{g(t)}  +\int_M (\scal_{g(t)}-\kappa)\cdot \hat u_R \,d\mathrm{vol}_{g(t)}\\
&\geq \Big[\langle\langle \scal_{g(t)}, u_R\rangle\rangle_{g(t)} -\langle\langle \scal_{g_0},  u_R\rangle\rangle_{g_0} \Big] 
+\kappa \int_M u_R \left(d\mathrm{vol}_{g_0}-d\mathrm{vol}_{g(t)} \right)
\\
&\quad +\int_M (\scal_{g(t)}-\kappa)\cdot \hat u_R \,d\mathrm{vol}_{g(t)}\\
    &=:\mathbf{I}+\mathbf{II}+\mathbf{III},
\end{split}
\end{equation}
where we have used $
\langle\langle \scal_{g_0}-\kappa ,\tilde u_R\rangle\rangle_{g_0}\geq 0$ from the assumption, as $u_R$ is compactly supported. Thanks to Lemma~\ref{lma:dist-monot}, $(\mathrm{scal}_{g(t)}-\kappa)\cdot u$ is integrable for each $t\in (0,t_0]$. By dominated convergence Theorem, $\mathbf{III}\to 0$ as $R\to+\infty$ or, equivalently, $m\to+\infty$.

It remains to estimate $\mathbf{I}$ and $\mathbf{II}$. We write $R_i:=2^i R_0$ for $i\geq 0$, so that $R=R_m$. In the following, we will use $C_i$ to denote any constants which depend only on $n,h,\Lambda, f, \kappa,t_0$, but not on $R_0$ and $m$. We will write $g$ instead of $g(t)$ for convenience. 

 Since $\varphi$ is compactly supported on $\{x: \rho\leq 2R\}\subseteq B_h(x_0,2R)$,  it is clear that 
\begin{equation}\label{eqn:II}
\begin{split}
|\mathbf{II}|&\leq C_1 \int_{B_h(x_0,2R)} u\cdot  |g-g_0| \,d\mathrm{vol}_h.
\end{split}
\end{equation}

For $\mathbf{I}$, we  use the distributional formulation from Definition~\ref{defn:R-lower-dist}:
\begin{equation}
\begin{split}
\mathbf{I}&=\int_M \left\langle V_{g_0},\tilde  \nabla \left( u_R \cdot\frac{d\mathrm{vol}_{g_0}}{d\mathrm{vol}_h}\right)\right\rangle_h-\left\langle V_{g},\tilde  \nabla \left( u_R \cdot \frac{d\mathrm{vol}_{g}}{d\mathrm{vol}_h}\right)\right\rangle_h\; d\mathrm{vol}_h\\
&\quad +\int_M \left( F_{g}  u_R \cdot \frac{d\mathrm{vol}_{g}}{d\mathrm{vol}_h}- F_{g_0}  u_R \cdot \frac{d\mathrm{vol}_{g_0}}{d\mathrm{vol}_h}\,\right)\; d\mathrm{vol}_h.
\end{split}
\end{equation}

By further linearizing each term, we see that 
\begin{equation}\label{eqn:I-leqII+III.}
\begin{split}
|\mathbf{I}|&\leq \mathbf{IV}+\mathbf{V}+\mathbf{VI}+\mathbf{VII},
\end{split}
\end{equation}
where 
\begin{equation}
\left\{
\begin{array}{ll}
\displaystyle \mathbf{IV}:= \int_M |V_{g_0}-V_{g}|\left| \tilde \nabla\left( u_R \cdot \frac{d\mathrm{vol}_{g}}{d\mathrm{vol}_h}\right)\right|\,d\mathrm{vol}_h,\\[4mm]
\displaystyle \mathbf{V}:= \int_M |V_{g}|\left| \tilde \nabla\left[ u_R \cdot \left(\frac{d\mathrm{vol}_{g}}{d\mathrm{vol}_h}-\frac{d\mathrm{vol}_{g_0}}{d\mathrm{vol}_h} \right)\right]\right|\,d\mathrm{vol}_h,\\[4mm]
\displaystyle \mathbf{VI}:= \int_M |F_{g}-F_{g_0}|  u_R \,d\mathrm{vol}_{g},\\[4mm]
\displaystyle \mathbf{VII}:= \int_M |F_{g_0}|    u_R\cdot \left|\frac{d\mathrm{vol}_{g}}{d\mathrm{vol}_h} -\frac{d\mathrm{vol}_{g_0}}{d\mathrm{vol}_h} \right| \,d\mathrm{vol}_{h}.
\end{array}
\right.
\end{equation}

From \eqref{eqn:F-V-dist}, it is clear that 
\begin{equation}\label{eqn:IV}
\begin{split}
\mathbf{IV}&\leq C_2\int_{B_h(x_0,2R)} \left[ |g-g_0||\tilde\nabla g|+ |\tilde \nabla(g-g_0)| \right] \cdot \\
&\hspace{4em} \left[ |\tilde \nabla u| +u |\tilde\nabla \varphi_R|+ u |\tilde\nabla g| \right] \,d\mathrm{vol}_h\\
&\leq C_2\int_{B_h(x_0,2R)} \left(  |g-g_0||\tilde\nabla g|+ |\tilde \nabla(g-g_0)|\right)\cdot  \left( |\tilde \nabla u|+ u|\tilde \nabla g|\right)\,d\mathrm{vol}_h\\
&\quad +C_3R^{-1}  \int_{ \{x: R\leq \rho\leq 2R\} }   u \left(  |g-g_0||\tilde\nabla g|+ |\tilde \nabla(g-g_0)|\right)\,d\mathrm{vol}_h\\
&=: \mathbf{IV}_1+\mathbf{IV}_2.
\end{split}
\end{equation}

Similarly, we have
\begin{equation}\label{eqn:V}
\begin{split}
\mathbf{V}&\leq C_4\int_{B_h(x_0,2R)} |\tilde\nabla g|\left( |\tilde\nabla u| |g-g_0|+  u|\tilde\nabla \varphi_R| |g-g_0|+u |\tilde\nabla (g-g_0)| \right)\,d\mathrm{vol}_h\\
&\leq C_4\int_{B_h(x_0,2R)} |\tilde \nabla g| \left( |\tilde\nabla u| |g-g_0|+u |\tilde\nabla (g-g_0)| \right)\,d\mathrm{vol}_h\\
&\quad +C_5 R^{-1}\int_{\{x:R\leq \rho\leq 2R \}} u|\tilde\nabla g| |g-g_0|\,d\mathrm{vol}_h\\
&=: \mathbf{V}_1+\mathbf{V}_2,
\end{split}
\end{equation}
and 
\begin{equation}\label{eqn:VI+VII} 
\begin{split}
\mathbf{VI}+\mathbf{VII} &\leq C_6\int_{B_h(x_0,2R)} u \Big[|g-g_0|+|g-g_0| |\tilde \nabla g|^2 \\
&\hspace{3em} + |\tilde \nabla g||\tilde \nabla (g-g_0)|+|\tilde \nabla g_0||\tilde \nabla (g-g_0)|\Big]
\,d\mathrm{vol}_h.
\end{split}
\end{equation}

By the Cauchy-Schwarz inequality and combining  \eqref{eqn:II}, \eqref{eqn:I-leqII+III.}, \eqref{eqn:IV}, \eqref{eqn:V} and \eqref{eqn:VI+VII}, we conclude that
\begin{equation}\label{eqn:main-error}
\begin{split}
|\mathbf{I}+\mathbf{II}|&\leq (\mathbf{IV}_1+\mathbf{IV}_2)+(\mathbf{V}_1+\mathbf{V}_2)+\mathbf{VI}+\mathbf{VII}+|\mathbf{II}|
\\
&\leq C_7 \int_{B_h(x_0,2R)} |g-g_0|\left(u+u|\tilde\nabla g|^2+|\tilde\nabla u|+|\tilde\nabla u| |\tilde\nabla g| \right) \,d\mathrm{vol}_h\\
&\quad + C_7\int_{B_h(x_0,2R)} |\tilde \nabla (g-g_0)|\left( u|\tilde \nabla g|+ u|\tilde \nabla g_0|+|\tilde\nabla u| \right)\,d\mathrm{vol}_h \\
&\quad +\mathbf{IV}_2+\mathbf{V}_2.
\end{split}
\end{equation}

\begin{claim}\label{claim:error-dist}
There exists $C(n,h,\Lambda,f,\kappa,t_0)>0$ such that for all $t\in (0,t_0]$ and $R\to+\infty$, we have
$$\mathbf{IV}_2+\mathbf{V}_2\leq CR^{-1}.$$ 
In particular, 
\begin{equation}\label{eqn:main-error-2}
\begin{split}
|\mathbf{I}+\mathbf{II}|&\leq C_7 \int_{M} |g-g_0|\left(u+u|\tilde\nabla g|^2+|\tilde\nabla u|+|\tilde\nabla u| |\tilde\nabla g| \right) \,d\mathrm{vol}_h\\
&\quad + C_7\int_{M} |\tilde \nabla (g-g_0)|\left( u|\tilde \nabla g|+|\tilde\nabla u| \right)\,d\mathrm{vol}_h+CR^{-1}.
\end{split}
\end{equation}

\end{claim}
\begin{proof}[Proof of Claim~\ref{claim:error-dist}]

We start with $\mathbf{IV}_2$. Recall from \eqref{eqn:IV} that 
\begin{equation}
\begin{split}
\mathbf{IV}_2&=C_3R^{-1}  \int_{ \{x: R\leq \rho\leq 2R\} }   u \left(  |g-g_0| |\tilde\nabla g|+ |\tilde \nabla(g-g_0)|\right)\,d\mathrm{vol}_h.
\end{split}
\end{equation}

From the construction of $\rho$, we have 
\begin{equation}
 \{x: R\leq \rho\leq 2R\} \subseteq A_h(x_0,10^{-1}R,2R):=B_h(x_0,2R)\setminus \overline{B_h(x_0,10^{-1}R)},
\end{equation}
for $R\to+\infty$.  We take a maximal $1$-net $\{B_h(x_i,1)\}_{i=1}^N$ of the annulus so that 
\begin{equation}
\coprod_{i=1}^N B_h\left(x_i,\frac18\right)\subseteq A_h(x_0,10^{-1}R,2R) \subseteq \bigcup_{i=1}^N B_h(x_i,1)
\end{equation}
and hence $N\leq  \exp(C_{10}R)$, by Bishop-Gromov volume comparison. 
 Hence,
\begin{equation}\label{eqn:cutoff-I}
\begin{split}
 &\quad \int_{ \{x: R\leq \rho\leq 2R\} }   u|\tilde \nabla(g-g_0)|\,d\mathrm{vol}_h\\
 &\leq \sum_{i=1}^N \int_{B_h(x_i,1)} u|\tilde \nabla(g-g_0)|\,d\mathrm{vol}_h\\
 &\leq \sum_{i=1}^N \left(\int_{B_h(x_i,1)} |\tilde \nabla(g-g_0)|^n\,d\mathrm{vol}_h\right)^\frac{1}{n}\cdot \left(\int_{B_h(x_i,1)} u^\frac{n-1}{n} \,d\mathrm{vol}_h\right)^\frac{n-1}{n}\\
 &\leq 4\e_1^2\cdot  \sum_{i=1}^N |\mathrm{Vol}_h(x_i,1)|^\frac{1}{n}\cdot \left(\int_{B_h(x_i,1)} u^\frac{n-1}{n} \,d\mathrm{vol}_h\right)^\frac{n-1}{n}.
\end{split}
\end{equation}

We now estimate the integral over each covering ball. Since $d_h(x_i,x_0)\geq 10^{-1}R\to+\infty$, $\e\leq (\frac{4(n-1)}{n})^{-1}$ and $f$ is compactly supported, Proposition~\ref{prop:Lp-u} implies that for each $1\leq i\leq N$,
\begin{equation}\label{eqn:local-Lp-covering}
\fint_{B_h(x_i,1)}u^p d\mathrm{vol}_{g(t)}\leq C_{11} \cdot \exp\left( -\frac{R^2}{C_{11}}\right).
\end{equation}

Together with $N\leq \exp(C_{10}R)$, we substitute  \eqref{eqn:local-Lp-covering} into \eqref{eqn:cutoff-I} to conclude 
\begin{equation}\label{eqn:IV-2term}
\begin{split}
\int_{ \{x: R\leq \rho\leq 2R\} }   u|\tilde \nabla(g-g_0)|\,d\mathrm{vol}_h\leq C_{11} \exp\left( -\frac{R^2}{C_{11}}+C_{10}R\right)\leq C_{12}
\end{split}
\end{equation}
for all $R>1$.

For the first term in $\mathbf{IV}_2$ and $\mathbf{V}_2$, we first recall from \eqref{eqn:V} that 
\begin{equation}
\mathbf{V}_2=C_5 R^{-1}\int_{\{x:R\leq \rho\leq 2R \}} u|\tilde\nabla g| |g-g_0|\,d\mathrm{vol}_h,
\end{equation}
which is a multiple of the first term in $\mathbf{IV}_2$. We estimate as in \eqref{eqn:cutoff-I} that 
\begin{equation}
    \begin{split}
       &\quad  \int_{\{x:R\leq \rho\leq 2R \}} u|\tilde\nabla g| |g-g_0|\,d\mathrm{vol}_h\\
        &\leq \sum_{i=1}^N \left(\int_{B_h(x_i,1)} |\tilde \nabla g|^n|g-g_0|^n\,d\mathrm{vol}_h\right)^\frac{1}{n}\cdot \left(\int_{B_h(x_i,1)} u^\frac{n-1}{n} \,d\mathrm{vol}_h\right)^\frac{n-1}{n}\\
 &\leq C_{13}\e_1^2\cdot  \sum_{i=1}^N |\mathrm{Vol}_h(x_i,1)|^\frac{1}{n}\cdot \left(\int_{B_h(x_i,1)} u^\frac{n-1}{n} \,d\mathrm{vol}_h\right)^\frac{n-1}{n}.
    \end{split}
\end{equation}

This yields 
\begin{equation}\label{eqn:IV-1term}
\begin{split}
\int_{ \{x: R\leq \rho\leq 2R\} }   u|\tilde \nabla g|| g-g_0|\,d\mathrm{vol}_h\leq C_{14},
\end{split}
\end{equation}
which is analogous to \eqref{eqn:IV-2term}.

\end{proof}

We next consider those terms appearing in \eqref{eqn:main-error-2} in which $u$ is involved. We use $\hat C_i$ to denote constants which also depend on $R_0$.
\begin{claim}\label{claim:term-with-u}
We have 
\begin{equation*}
\begin{split}
\lim_{t\to0}\int_{M} u \big(&|g-g_0|+|g-g_0||\tilde \nabla g|^2\\
&+ |\tilde\nabla g||\tilde \nabla (g-g_0)|+ |\tilde\nabla g_0||\tilde \nabla (g-g_0)| \big)\,d\mathrm{vol}_h=0.
\end{split}
\end{equation*}
\end{claim}
\begin{proof}[Proof of Claim~\ref{claim:term-with-u}]

To begin, we first decompose the domain of integration into $B_h(x_0,2R_0)\cup \bigcup_{i=1}^{\infty} A_h(x_0,2R_i,2R_{i+1})$.
Following the proof of Claim~\ref{claim:error-dist}, we take a maximal $1$-net $\{B_h(x_k,1)\}_{k=1}^N$ of $A_h(x_0,2R_i,R_{i+1})$ so that $N\leq \exp(C_{10}R_i)$. The leading term is controlled as follows:
\begin{equation}
\begin{split}
&\quad \int_{B_h(x_0,2R_0)} u |g-g_0||\tilde\nabla g|^2 d\mathrm{vol}_h\\
&\leq \left(\int_{B_h(x_0,2R_0)}|\tilde\nabla g|^n d\mathrm{vol}_h\right)^\frac{2}{n}\left(\int_{B_h(x_0,2R_0)} u^\frac{n}{n-2}|g-g_0|^\frac{n}{n-2} d\mathrm{vol}_h\right)^\frac{n-2}{n}\\
&\leq \hat C_1 \left(\int_{B_h(x_0,2R_0)} u^\frac{2n}{n-2} d\mathrm{vol}_h\right)^\frac{n-2}{2n}\left(\int_{B_h(x_0,2R_0)} |g-g_0|^\frac{2n}{n-2} d\mathrm{vol}_h\right)^\frac{n-2}{2n}\\
&\leq \hat C_2 \left(\int_{B_h(x_0,2R_0)} |g-g_0|^\frac{2n}{n-2} d\mathrm{vol}_h\right)^\frac{n-2}{2n}=o(1)
\end{split}
\end{equation}
as $t\to 0$ for the fixed $R_0>0$, where we have used $\frac{2n}{n-2}<(4\e)^{-1}$ and Proposition~\ref{prop:Lp-u}. We now consider the annulus $A_h(x_0,2R_i,2R_{i+1})$, which is covered by $\{B_h(x_k,1)\}_{k=1}^N$. We argue similarly that for $1\leq k\leq N$, 
\begin{equation}
\begin{split}
&\quad \int_{B_h(x_k,1)} u |g-g_0||\tilde\nabla g|^2 d\mathrm{vol}_h\\
&\leq \e_1^2  \left(\int_{B_h(x_k,1)} u^\frac{2n}{n-2} d\mathrm{vol}_h\right)^\frac{n-2}{2n}\left(\int_{B_h(x_k,1)} |g-g_0|^\frac{2n}{n-2} d\mathrm{vol}_h\right)^\frac{n-2}{2n}\\
&\leq C_{13}\exp\left(-C_{13}^{-1} R_i^2 \right),
\end{split}
\end{equation}
where we have used Proposition~\ref{prop:Lp-u} as well as the fact that $d_h(x,x_0)\geq R_i$ for $x_k\in A_h(x_0,R_i,R_{i+1})$. Hence, 
\begin{equation}
\begin{split}
&\quad \int_{M\setminus B_h(x_0,2R_0)}u |g-g_0||\tilde\nabla g|^2 d\mathrm{vol}_h \\
&\leq \sum_{i=1}^\infty\int_{A_h(x_0,2R_i,2R_{i+1})} u |g-g_0||\tilde\nabla g|^2 d\mathrm{vol}_h\\
&\leq \sum_{i=1}^\infty C_{13}\exp\left(-C_{13}^{-1} R_i^2 +C_{10}R_i\right)\leq R_0^{-1}
\end{split}
\end{equation}
as $R_0\to+\infty$. The first term involving solely $u|g-g_0|$ can be handled similarly (and in a simpler fashion).  

For the last term, we similarly apply Young's inequality to deduce
\begin{equation}
\begin{split}
&\quad \int_{B_h(x_0,2R_0)} u|\tilde \nabla g||\tilde \nabla(g-g_0)|\,d\mathrm{vol}_h\\
&\leq  \left(\int_{B_h(x_0,2R_0)} |\tilde \nabla(g-g_0)|^n\,d\mathrm{vol}_h\right)^{\frac1n} \cdot \left(\int_{B_h(x_0,2R_0)} |\tilde\nabla g|^n d\mathrm{vol}_h\right)^\frac{1}{n}\\
&\quad \cdot \left(\int_{B_h(x_0,2R_0)}u^\frac{n}{n-2}d\mathrm{vol}_h \right)^\frac{n-2}{n}\\
&\leq \hat C_3\left(\int_{B_h(x_0,2R_0)} |\tilde \nabla(g-g_0)|^n\,d\mathrm{vol}_h\right)^{\frac1n}=o(1)
\end{split}
\end{equation}
as $t\to0$ for fixed $R_0>0$, 
by Proposition~\ref{prop:Lp-u} and $\frac{n}{n-2}<(4\e)^{-1}$. Similarly, using a covering argument, we obtain that 
\begin{equation}
\begin{split}
&\quad \int_{M\setminus B_h(x_0,2R)} u|\tilde \nabla g||\tilde \nabla(g-g_0)|\,d\mathrm{vol}_h\\
&\leq \sum_{i=1}^\infty \int_{A_h(x_0,2R_i,2R_{i+1})} u|\tilde \nabla g||\tilde \nabla(g-g_0)|\,d\mathrm{vol}_h\\
&\leq  \sum_{i=1}^\infty \sum_{k=1}^{N_i} \left(\int_{B_h(x_k,1)} |\tilde \nabla(g-g_0)|^n\,d\mathrm{vol}_h\right)^{\frac1n}\left(\int_{B_h(x_k,1)} |\tilde\nabla g|^n d\mathrm{vol}_h\right)^\frac{1}{n}\\
&\quad \cdot \left(\int_{B_h(x_k,1)}u^\frac{n}{n-2}d\mathrm{vol}_h \right)^\frac{n-2}{n}\\
&\leq (2\e_1^2)\sum_{i=1}^\infty \sum_{k=1}^{N_i} \left(\int_{B_h(x_k,1)}u^\frac{n}{n-2}d\mathrm{vol}_h \right)^\frac{n-2}{n}\leq R_0^{-1}
\end{split}
\end{equation}
as $R_0\to+\infty$. By letting $t\to0$ followed by $R_0\to+\infty$, we prove the claim.
\end{proof}

Now we handle the remaining terms involving $|\tilde\nabla u|$ in \eqref{eqn:main-error-2}.
\begin{claim}\label{claim:term-with-du}
We have 
\begin{equation}
\begin{split}
\lim_{t\to 0}\int_{M} |\tilde\nabla u| \left(|g-g_0|+|g-g_0||\tilde \nabla g|+ |\tilde \nabla (g-g_0)| \right)\,d\mathrm{vol}_h=0.
\end{split}
\end{equation}
\end{claim}
\begin{proof}

By Proposition~\ref{prop:Wp-u} and $\frac{n-2}{n+2}<\frac1{4\e}$,
\begin{equation}
\begin{split}
&\quad \int_{B_h(x_0,2R_0)} |\tilde\nabla u||g-g_0| |\tilde\nabla g|d\mathrm{vol}_h\\
&\leq \left( \int_{B_h(x_0,2R_0)} |\tilde\nabla u|^2  d\mathrm{vol}_h\right)^\frac12 \left(\int_{B_h(x_0,2R_0)}|\tilde\nabla g|^n\,d\mathrm{vol}_h \right)^\frac1n\\
&\quad \cdot \left( \int_{B_h(x_0,2R_0)}|g-g_0|^\frac{2n}{n-2}\,d\mathrm{vol}_h \right)^\frac{2n}{n-2}\\
&\leq \hat C_4 ||g-g_0||_{W^{1,n}(B_h(x_0,2R_0)}=o(1)
\end{split}
\end{equation}
as $t\to0$ for fixed $R_0>0$. We now handle the remaining set $M\setminus B_h(x_0,2R_0)=\bigcup_{i=1}^\infty A_h(x_0,2R_i,2R_{i+1})$ using an argument similar to the proof of Claim~\ref{claim:term-with-u}. To that end, we let $\{B_h(x_k,1)\}_{k=1}^{N_i}$ be a covering of $A_h(x_0,2R_i,2R_{i+1})$, where $N_i\leq \exp(C_{10}R_i)$. Then 
\begin{equation}
\begin{split}
&\quad \int_{M\setminus B_h(x_0,2R_0)} |\tilde\nabla u||g-g_0| |\tilde\nabla g|d\mathrm{vol}_h\\
&\leq \sum_{i=1}^\infty\sum_{k=1}^{N_i}\int_{B_h(x_k,1)}|\tilde\nabla u||g-g_0| |\tilde\nabla g|d\mathrm{vol}_h\\
&\leq \sum_{i=1}^\infty\sum_{k=1}^{N_i}\left( \int_{B_h(x_k,1)} |\tilde\nabla u|^2  d\mathrm{vol}_h\right)^\frac12 \left(\int_{B_h(x_k,1)}|\tilde\nabla g|^n\,d\mathrm{vol}_h \right)^\frac1n\\
&\quad \quad \quad  \quad  \quad  \cdot \left( \int_{B_h(x_k,1)}|g-g_0|^\frac{2n}{n-2}\,d\mathrm{vol}_h \right)^\frac{2n}{n-2}\\
&\leq C_{14}\sum_{i=1}^\infty\sum_{k=1}^{N_i} \exp(-C_{14}^{-1}R_i^2+C_{10}R_i)\leq R_0^{-1}
\end{split}
\end{equation}
as $R_0\to+\infty$, using Proposition~\ref{prop:Wp-u}.

The first term is similar, and simpler. For the last term, we argue similarly using Proposition~\ref{prop:Wp-u} and $n\geq 3$ that 
\begin{equation}
\begin{split}
&\quad \int_{B_h(x_0,2R_0)} |\tilde\nabla u| |\tilde\nabla (g-g_0)|d\mathrm{vol}_h\\
&\leq  \left(\int_{B_h(x_0,2R_0)} |\tilde\nabla u|^2 d\mathrm{vol}_h\right)^\frac12\left(\int_{B_h(x_0,2R_0)}  |\tilde\nabla (g-g_0)|^2 d\mathrm{vol}_h\right)^\frac12\\
&\leq \hat C_5 ||g-g_0||_{W^{1,n}(B_h(x_0,2R_0)}=o(1)
\end{split}
\end{equation}
as $t\to0$ for fixed $R_0>0$. In order to integrate over $M\setminus B_h(x_0,2R_0)$, we argue similarly using a covering argument to deduce that 
\begin{equation}
\begin{split}
&\quad \int_{M\setminus B_h(x_0,2R_0)} |\tilde\nabla u| |\tilde\nabla (g-g_0)|d\mathrm{vol}_h\leq R_0^{-1}
\end{split}
\end{equation}

This proves the claim by letting $t\to0$ followed by $R_0\to+\infty$.
\end{proof}

We now combine \eqref{eqn:scal-gt} and \eqref{eqn:scal-gt-1} with Claim~\ref{claim:error-dist}, Claim~\ref{claim:term-with-u} and Claim~\ref{claim:term-with-du} and let $m\to+\infty$ followed by $t\to0$ to show that 
\begin{equation}
\int_M (\scal_{g(t_0)}-\kappa)\cdot f d\mathrm{vol}_{g(t_0)}\geq 0
\end{equation}
for any compactly supported smooth non-negative function $f$. Since $g(t)$ is smooth and $t_0$ is arbitrary, this implies $\scal_{g(t)}(x)\geq \kappa$ on $M\times (0,T]$.
\end{proof}

 Theorem~\ref{thm:smoothing} is now a direct consequence of Theorem~\ref{thm:RDF-preserved} and Theorem~\ref{thm:RDF-exist}.
\begin{proof}[Proof of Theorem~\ref{thm:smoothing}]
By \cite[Proposition 3.1]{ChuLee2025}, there exists $r_0>0$ such that $\hat g_0:=r_0^{-2}g_0$ satisfies the assumptions of Theorem~\ref{thm:RDF-exist} with respect to $\hat h:=r_0^{-2}h$ and $\scal(\hat g_0)\geq \hat \kappa :=\kappa r_0^2$ on $M$ in the distributional sense. Let $\hat g(t)$, $t\in (0,\hat T]$ be the Ricci–DeTurck $\hat h$-flow obtained from applying Theorem~\ref{thm:RDF-exist} to $\hat g_0$. Theorem~\ref{thm:RDF-preserved} implies that $\scal(\hat g(t))\geq \hat\kappa$ on $(0,\hat T]$. The theorem then follows by scaling back. 
\end{proof}

\section{Rigidity under distributional scalar curvature lower bound}\label{sec:rigidty}

In this section, we discuss the application of Theorem~\ref{thm:RDF-preserved} to the rigidity problem in scalar curvature for metrics with regularity $L^\infty\cap W^{1,n}$ with respect to some smooth metric with bounded geometry. We start with the complete non-compact case, i.\,e., the positive mass theorem. The proof of the torus rigidity is similar and simpler. The main idea is to regularize the metric while preserving the scalar curvature geometry, which has already been taken care of in Theorem~\ref{thm:RDF-preserved}.

\begin{proof}[Proof of Theorem~\ref{thm:PMT}] 
The proof of the non-negativity of mass follows the strategy in the proof of \cite[Theorem 18]{McFeronSzekelyhidi2012}, except we use Theorem~\ref{thm:RDF-preserved} to preserve the non-negative scalar curvature. We include it for the reader's convenience.

We choose a background metric $h$ such that $h$ is Euclidean outside a compact set and hence $g_0$ is globally in $W^{1,n}$ with respect to $h$. As in the proof of Theorem~\ref{thm:smoothing}, we may assume by scaling that $g_0$ admits a unique Ricci–DeTurck $h$-flow $g(t)$ on $M\times (0,T]$ which satisfies 
\begin{enumerate}
\item[(i)] $\Lambda^{-1}h\leq g(t)\leq \Lambda h$,
\item[(ii)] $|\tilde\nabla^2 g(t)|+|\tilde\nabla g(t)|^2\leq \e t^{-1}$,
\item[(iii)] $\scal(g(t))\geq 0$,
\item[(iv)] $g(t)\to g_0$ in $L^p_{\loc}$ as $t\to0$ for all $p>0$,
\end{enumerate}
for some $\Lambda>1$. Here $\e > 0$ can be arbitrarily small, provided that we reduce the existence time further as needed. For the purpose of our application, it will be more transparent to keep track of the construction of $g(t)$ from \cite{ChuLee2025}. 

Let $g_{i,0}$ be a sequence of smooth metrics converging to $g_0$ in $W^{1,n}_{\loc}$. We may assume $g_{i,0}=g_0$ outside of a compact set, since $g_0$ is smooth at infinity.  By scaling, we may assume that $g_{i,0}$ admits a short-time solution $g_i(t)$, $t\in (0,T]$ to the Ricci–DeTurck $h$-flow. Since $g_0$ (and hence $g_{i,0}$) is asymptotically flat in the classical $C^2_\delta$-sense outside a compact set, by the local estimate of Ricci–DeTurck $h$-flow \cite[Proposition 2.2]{ChuLee2025}, $g_i(t)$ is also smooth outside a compact set up to $t=0$. It follows from \cite[Theorem 5]{McFeronSzekelyhidi2012} (see also \cite[Theorem 2.2]{Li2018}) that $g_i(t)$ is also asymptotically flat in the classical $C^2_\delta$-sense, uniformly in $i$. 

For each $i$, there is a diffeomorphism $\Phi_i(t)$ starting from $\Phi_i(0)=\mathrm{Id}$ such that $\Phi_i^*g_i(t)$ is a solution to Ricci flow so that \cite[Corollary 12]{McFeronSzekelyhidi2012} applies to yield $m_{\ADM}(g_{0})=m_{\ADM}(g_{i,0})= m_{\ADM}(\Phi_i^*g_i(t))$ for all $t\in [0,T]$ in view of $g_{i,0}=g_0$ outside a compact set. Since the mass is independent of the choice of asymptotic coordinates under our decay conditions by \cite[Theorem 4.2]{Bartnik}, we conclude that $m_{\ADM}(g_{0})= m_{\ADM}(g_i(t))$ for all $t\in [0,T]$. Furthermore, by \cite[Theorem 14 and Lemma 16, 17]{McFeronSzekelyhidi2012} and the construction of $g(t)=\lim_{i\to+\infty}g_i(t)$, we conclude that $g(t)$ is in $C^{1,\a}_{\delta'}$ for some $\delta'>\frac{n-2}2$ and $m_{\ADM}(g_0)\geq m_{\ADM}(g(t))$. Since $\scal(g(t))\geq 0$ by Theorem~\ref{thm:RDF-preserved}, it follows from \cites{SchoenYauPMT,SchoenNOTE,LeeParker} that 
\begin{equation}
    m_{\ADM}(g_0)\geq m_{\ADM}(g(t))\geq 0.
\end{equation}

It remains to prove the rigidity. If  $m_{\ADM}(g_0)=0$, then $m_{\ADM}(g(t))=0$ for $t\in (0,T]$. The smooth rigidity in the positive mass theorem \cites{SchoenYauPMT,SchoenNOTE,LeeParker} implies that $(M,g(t))$ is isometric to $\mathbb{R}^n$ for all $t\in (0,T]$. We consider the ordinary differential equation in \eqref{eqn:W-defn}:
\begin{equation}\label{eqn:OODE}
    \left\{
    \begin{array}{ll}
      \partial_t \Psi_t(x)=-W(\Psi_t(x),t),\\[3pt]
    \Psi_{T}(x)=x,
   \end{array}
    \right.
\end{equation}
so that $\Psi_t^*g(t)=g(T):=\hat g$ by flatness. Let $D$ be the connection of the pull-back bundle induced by the smooth map $\Psi_t:(\mathbb{R}^n,\hat g)\to (\mathbb{R}^n,g(t))$. By differentiating $\Psi_t^*g(t)=\hat g$, we have 
\begin{equation}
    (\Psi_t)_{ik}^\a (\Psi_t)^\b_j g_{\a\b}+ (\Psi_t)_{i}^\a (\Psi_t)^\b_{kj} g_{\a\b}=D_k\hat g_{ij}=0
\end{equation}
for all $i,j,k$. By interchanging $i,k$, we conclude that 
\begin{equation}
   (\Psi_t)_{k}^\a (\Psi_t)^\b_{ij} g_{\a\b}=- (\Psi_t)_{ik}^\a (\Psi_t)^\b_j g_{\a\b}=(\Psi_t)_{i}^\a (\Psi_t)^\b_{kj} g_{\a\b}.
\end{equation}

Since $i,j,k$ are arbitrary, we deduce that $D_iD_j \Psi_t^\a=0$. In local coordinates, we thus have 
$$ 0=D_iD_j \Psi_t^\a={\partial_i\partial_j\Psi_t^\a}-\Gamma_{ij}^k \cdot\partial_k \Psi_t^\a+\hat \Gamma_{\b\gamma}^\a \cdot \partial_i\Psi_t^\b\cdot  \partial_j\Psi_t^\gamma.$$
Note that $\Psi_t$ is locally bounded in the $W^{2,n}$-norm as $t\to0^+$. By the Rellich–Kondrachov theorem, $\Psi_t$ converges to some map $\Psi_0$ in $C^\a_{\loc}$ for any $\a\in (0,1)$. Since $\Psi_t$ is surjective, it follows from the convergence that $\Psi_0$ is also surjective. In order to show that $\Psi_0$ is injective, we observe that 
\begin{equation}
\begin{split}
d_{\hat g}(x,y)&=d_{g(t)}(\Psi_t(x),\Psi_t(y))\leq \Lambda^{1/2}d_h(\Psi_t(x),\Psi_t(y))
\end{split}
\end{equation}
for any $x,y\in M$. The injectivity of $\Psi_0$ then follows by letting $t\to0$. By the local regularity of Ricci–DeTurck flow \cite[Proposition 2.2]{ChuLee2025}, the convergence is in the locally smooth topology outside a compact set. This proves the flatness of $g_0$ outside of a compact set.

Finally, we prove the isometry in distance in the sense of Definition~\ref{defn:intr-dist}, that is,
\begin{equation}
d_{\hat g}(x,y)=d_{g_0}\left(\Psi_0(x),\Psi_0(y) \right)
\end{equation}
for all $x,y\in M$. This follows after suitably modifying the argument in \cite[Proposition 8.11]{LeeNaberNeumayer}, using also the fact that $g(t)$ is isometric to Euclidean space. Since we are in a simpler setting, we include the proof for the reader's convenience. We first note that from metric equivalence, it is clear that 
\begin{equation}
\Lambda^{-\frac{n+p}{2p}}\cdot d_{h,p}(x,y)\leq d_{g(t),p}(x,y)\leq \Lambda^\frac{n+p}{2p}\cdot d_{h,p}(x,y).
\end{equation}
In particular, we may assume $x,y\in B_{h}(x_0,R_0)$ for some $R_0(n,p,x_0,x,y,h)>0$, by \cite[Example 2.30]{LeeNaberNeumayer}.

For $p>n$, $\mu>1$ and $\e>0$, we let $f$ be a function such that $f(y)=0$, $||\nabla^{g_0} f||_{L^{p\mu}(M,g_0)}\leq 1$ and 
\begin{equation}
d_{g_0,p\mu}(x,y)\leq \e+ |f(x)-f(y)|.
\end{equation}
Such an $f$  exists in view of the definition of $d_{g_0,p\mu}$  given in \eqref{alternativddefn}
We replace $f$ by $\min\{ |f|,4R_0\}$.

We let $R>R_0$ and $\varphi_R$ be a smooth cutoff function such that $\varphi=1$ on $B_h(x_0,R)$, $\varphi$ vanishes outside $B_h(x_0,2R)$ and it satisfies $|\nabla ^h\varphi|\leq 10^4R^{-1}$. We replace the test function by $\hat f:= \varphi_R f$,
so that 
\begin{equation}\label{eqn:test-1}
\begin{split}
||\nabla^{g(t)} \hat f||_{L^p(M,g(t))}&\leq ||\varphi_R\nabla^{g(t)}  f||_{L^p(M,g(t))}+ ||f\nabla^{g(t)}  \varphi_R||_{L^p(M,g(t))}\\
&\leq ||\nabla^{g(t)}  f||_{L^p(B_h(x_0,2R),g(t))}+CR_0 R^{\frac{n}{p}-1}
\end{split}
\end{equation}
for some $C(n,\Lambda,h)>0$. Here we have used the fact that $h$ is of Euclidean volume growth and the metric equivalence between $g(t)$ and $h$.

We now establish control on the first term. Denote the conjugate exponent of $\mu$ by $\mu'$. Then for $t$ sufficiently small (depending also on $R$), we have $||g(t)-g_0||_{L^{\mu'}(B_h(x_0,2R))}<\e$ and hence 
\begin{equation}\label{eqn:test-2}
\begin{split}
\int_{B_h(x_0,2R)} |\nabla^{g(t)} f|^{p} d\mathrm{vol}_{g(t)}
&\leq \int_{B_h(x_0,2R)} (1+C_{n,p}|g-g_0|)|\nabla f|_{g_0}^{p} d\mathrm{vol}_{g_0}\\
&\leq \left|\mathrm{Vol}_{g_0}B_h(x_0,2R) \right|^{\frac1{\mu'}}+C_{n,p} \e,
\end{split}
\end{equation}
using $||\nabla f||_{L^{p\mu}(M,g_0)}\leq 1$. Combining \eqref{eqn:test-1} with \eqref{eqn:test-2} yields 
\begin{equation}
\begin{split}
||\nabla^{g(t)} \hat f||_{L^p(M,g(t))}\leq \left[  \left|\mathrm{Vol}_{g_0}B_h(x_0,2R) \right|^{\frac1{\mu'}}+C_{n,p} \e\right]^\frac1p+CR_0R^{\frac{n}{p}-1}
\end{split}
\end{equation}
and hence
\begin{equation}
\begin{split}
&\quad d_{g_0,p\mu}(x,y)-\e\\
&\leq d_{g(t),p}(x,y)\cdot\left( \left[  \left|\mathrm{Vol}_{g_0}B_h(x_0,2R) \right|^{\frac1{\mu'}}+C_{n,p} \e\right]^\frac1p+CR_0R^{\frac{n}{p}-1}\right)\\
&= d_{\hat g,p}(\Psi_t^{-1}(x),\Psi_t^{-1}(y))\cdot\left( \left[  \left|\mathrm{Vol}_{g_0}B_h(x_0,2R) \right|^{\frac1{\mu'}}+C_{n,p} \e\right]^\frac1p+CR_0R^{\frac{n}{p}-1}\right).
\end{split}
\end{equation}

Since $d_p$ is a continuous function on $M\times M$ by \cite[(2.13)]{DeCeccoPalmieri} (see also the discussion in \cite[Section 2]{LeeNaberNeumayer}) and $\Psi_0$ is a Hölder homeomorphism, we may let $t\to 0$ and $\e\to 0$ to conclude 
\begin{equation}
    d_{g_0,p\mu}(x,y)\leq  d_{\hat g,p}(\Psi_0^{-1}(x),\Psi_0^{-1}(y))\cdot\left(  \left|\mathrm{Vol}_{g_0}B_h(x_0,2R) \right|^{\frac1{p\mu'}}+CR_0R^{\frac{n}{p}-1}\right).
\end{equation}

Furthermore, in view of $\lim_{p\to+\infty}d_p=d_g$ by \cite[Theorem 2.6]{DeCeccoPalmieri}, we may then let $p\to+\infty$ and $R\to +\infty$ to conclude 
$$d_{g_0}(x,y)\leq d_{\hat g}(\Psi_0^{-1}(x),\Psi_0^{-1}(y))$$
for all $x,y\in M$. The reverse inequality can be proved by interchanging the roles of $g_0$ and $g(t)$. This shows that $\Psi_0$ is an isometry with respect to the distance in the sense of Definition~\ref{defn:intr-dist}. This completes the proof.
\end{proof}

The compact case is a straightforward adaption of the proof of rigidity in the positive mass theorem, i.\,e., Theorem~\ref{thm:PMT}.
\begin{proof}[Proof of Theorem~\ref{thm:compact}]
By the proof of Theorem~\ref{thm:smoothing}, we may choose $h$ such that $g_0$ admits a short-time solution $g(t)$ to the Ricci–DeTurck flow on $\mathbb{T}^n\times (0,T]$ which satisfies 
\begin{enumerate}
\item[(i)] $\Lambda^{-1}h\leq g(t)\leq \Lambda h$,
\item[(ii)] $|\tilde\nabla^2 g(t)|+|\tilde\nabla g(t)|^2\leq \e t^{-1}$,
\item[(iii)] $\scal(g(t))\geq 0$,
\item[(iv)] $g(t)\to g_0$ in $L^p$ as $t\to0$ for all $p>0$,
\end{enumerate}
for some $\Lambda>1$.
Furthermore, $\e>0$ can be made arbitrarily small by reducing $T$ or re-choosing $h$ as required. By the classical torus rigidity \cites{SchoenYauTorus,SchoenYauTorusII,
GromovLawson1980}, $\mathrm{Rm}(g(t))\equiv 0$ for all $t\in (0,T]$. The rest of proof is identical to that of Theorem~\ref{thm:PMT}.
\end{proof}


\end{document}